\documentclass[11pt]{article}
\usepackage{fancyhdr}
\usepackage{epsf, times, amsmath, amssymb, graphicx, 
 makeidx, color, rotating, setspace, flafter, enumerate, wrapfig}
\usepackage[sort]{natbib}
\usepackage{a4wide}

 \usepackage[colorinlistoftodos, textwidth=2cm]{todonotes}

\usepackage[dvips]{epsfig}
\usepackage{amscd}
\usepackage{amssymb}
\usepackage{amsthm}
\usepackage{amsmath}
\usepackage{latexsym}

\usepackage{pstricks}
\usepackage{pst-plot}
\usepackage{pst-node}
\SpecialCoor

\usepackage{subfigure}
\theoremstyle{plain}

\theoremstyle{plain}
\newtheorem{lem}{Lemma}
\newtheorem{prop}{Proposition}
\newtheorem{cor}{Corollary}
\newtheorem{ex}{Example}
\newtheorem{as}{Assumption}
\theoremstyle{definition}
\newtheorem{defi}{Definition}
\newtheorem{rem}{Remark}

{%
\setcounter{enumi}{0}

\begin{enumerate}}%
{\end{enumerate}
}
{%
\setcounter{enumi}{0}

\begin{enumerate}}%
{\end{enumerate}
}

 \newcommand{\R}{\ensuremath{\Bbb{R}}}
 \newcommand{\E}{\ensuremath{\Bbb{E}}}
\newcommand{\Levy}{L\'{e}vy\ } 
\newcommand{\levy}{L\'{e}vy} 

\newcommand{\Ito}{It\^{o} }
\newcommand{\leb}{\ensuremath{\mathit Leb}}
\newcommand{\indicator}{\ensuremath{\mathbb{I}}}

\def\N{{I\!\!N}}

\begin{document}
\title{Recent advances in ambit stochastics with a view towards tempo-spatial stochastic volatility/intermittency}

\author{Ole E.\ Barndorff-Nielsen\footnote{
Thiele Center, 
Department of Mathematical Sciences
\&  CREATES,
 School of Economics and Management,
Aarhus University,
Ny Munkegade 118,
DK-8000 Aarhus C, Denmark,
{\tt oebn\@@imf.au.dk}}\\
Aarhus University 
\vspace{0.5cm}
 \and 
Fred Espen Benth\footnote{
 Centre of Mathematics for Applications,
 University of Oslo,
 P.O. Box 1053, Blindern,
 N-0316 Oslo, Norway
{\tt fredb\@@math.uio.no}}\\
University of Oslo 
\vspace{0.5cm}
\and Almut E.\ D.\ Veraart\footnote{
 Department of Mathematics,
Imperial College London,
180 Queen's Gate,
SW7 2AZ London, 
 UK,
 {\tt a.veraart\@@imperial.ac.uk}}\\
 Imperial College London \& CREATES
}

\maketitle
\begin{abstract}
 Ambit stochastics is the name for the theory and applications of \emph{ambit fields} and \emph{ambit processes} and constitutes a new research area in stochastics for tempo-spatial phenomena. This paper gives an overview of  the main findings in ambit stochastics up to date and establishes new results on general properties of ambit fields. Moreover, it develops the concept of tempo-spatial stochastic  volatility/intermittency   within ambit fields.
Various types of volatility modulation ranging from stochastic scaling of the amplitude, to stochastic time change and extended subordination of random measures and to probability and \Levy mixing of volatility/intensity parameters will be developed.
 Important examples for concrete model specifications within the class of ambit fields are given. 
\end{abstract}
{\bf Keywords:} Ambit stochastics, random measure, \Levy basis, stochastic volatility, extended subordination, meta-times, non-semimartingales.\newline \ \newline
{\bf MSC codes:} 60G10, 
60G51, 
60G57, 
60G60, 
\section{Introduction}
Tempo-spatial stochastic models describe objects  which are influenced both by time and  location. They naturally arise in various  applications such as  agricultural and environmental  studies,  ecology,  meteorology,
geophysics, turbulence,  biology,    global economies and financial markets. 
Despite the fact that the aforementioned areas of application are very different
in nature, they pose some common challenging mathematical and statistical problems.
While there is a very comprehensive literature on both time series modelling, see e.g.~\cite{Hamilton1994, BrockwellDavis2002}, and also on modelling purely spatial phenomena, see e.g.~\cite{Cressie1993}, tempo-spatial stochastic  modelling has only recently become  one of the most challenging research frontiers in modern probability and statistics, see \cite{Finkenstadtetal2007,CressieWikle2011} for textbook treatments.
Advanced and novel methods from statistics,  probability, and stochastic analysis  are called for to address the difficulties in constructing and estimating  flexible and, at the same time, parsimoniously parametrised  stochastic tempo-spatial  models.
There are various challenging issues which need to be addressed when dealing with tempo-spatial data, starting from data collection, model building, model estimation and selection, and model validation up to prediction. 
This paper focuses on building 
a flexible, dynamic tempo-spatial modelling framework, in which we   develop the novel concept of \emph{tempo-spatial stochastic volatility/intermittency} which allows one  to model random \emph{volatility clusters} and fluctuations both in time and in space. 
Note that \emph{intermittency} is an alternative name for \emph{stochastic volatility}, used in particular in turbulence. 
The presence of stochastic volatility is an empirical fact  in a variety of scientific fields (including the ones mentioned above), see e.g.~\cite{Amiri2009,Huangetal2010, ShephardAndersen2009}. Despite its ubiquitousness and importance, however, this important quantity has so far been often overlooked in
the tempo-spatial literature. Possibly this is  due to the fact that 
 stochastic volatility induces   high mathematical complexity which is challenging  
both  in terms of model building as well as for  model estimation.

The concept of stochastic volatility needs to be defined with respect to a base model, which we introduce in the following.
While the models should be the best realistic description of the underlying
random phenomena, they also have to be treatable for further use in the
design of controls, or risk evaluation, or planning of engineering equipment
in the areas of application in which the tempo-spatial phenomenon is considered.
Stochastic models for such tempo-spatial systems are typically formulated
in terms of evolution equations, and they rely on the use of random fields.
We  focus on such random fields  which are defined in
terms of certain types of stochastic integrals with respect to random measures which can be regarded as a unifying framework which encompasses many of the traditional modelling classes.
We will work with random fields denoted by
$Y_t({\bf x}) \in \mathbb{R}$  with $Y_t({\bf x}):= Y({\bf x},t)$, where $t \in \mathbb{R}$ denotes the  temporal parameter and ${\bf x} \in \mathbb{R}^d$ denotes the  spatial parameter, where $d \in \N_0$. Typically, we have $d\in \{1, 2, 3\}$   representing, for instance, longitude, latitude and height.  Note that by choosing \emph{continuous} parameters $({\bf x},t)$, we can later  allow for considerable flexibility in the discretisation of the model (including in particular  the possibility of irregularly spaced data). This maximises the potential for wide  applications of the model. 
We expect that the random variables $Y_t({\bf x})$ and $Y_{t'}({\bf x}')$ are  correlated as soon as the points $({\bf x},t)$ and $({\bf x}',t')$ are \lq\lq proximate\rq\rq\ according to a suitable measure of distance. This idea can be formalised in terms of a set $A_t({\bf x})\subseteq  \mathbb{R}^d\times \mathbb{R}$ such that for all  $({\bf x}',t') \in A_t({\bf x})$, the random variables 
 $Y_t({\bf x})$ and $Y_{t'}({\bf x}')$ are correlated. The set $A_t({\bf x})$ is sometimes referred to as \emph{causality cone} and more recently as \emph{ambit set}, see \cite{BNSch04}, describing the sphere of influence of the random variable $Y_t({\bf x})$. A concrete example of an ambit set would, for instance, be given by  a  
 light cone or a sound cone.

As the base model for a tempo-spatial object, we  choose \vspace{-0.1in}
\begin{align}\label{base}
Y_t({\bf x}) = \int_{A_t({\bf x})} h({\bf x},t; {\boldsymbol \xi}, s)  L(d{\boldsymbol \xi},ds),
\end{align}
where $L$ is an infinitely divisible, independently scattered random measure, 
i.e.~a L\'{e}vy basis. Under suitable regularity conditions, our base model \eqref{base} can be linked to solutions of certain types of stochastic partial differential equations, which are often used for tempo-spatial modelling, see  \cite{BNBV2011} for details.  The kernel function $h: \mathbb{R}^d\times \mathbb{R} \times \mathbb{R}^d\times \mathbb{R} \to \mathbb{R}$ needs to satisfy some integrability conditions to ensure the existence of the integral, which we will study in detail in Section \ref{Section:Background}. 
 Note that the covariance structure of the base model is fully determined by the choice of the kernel function $h$ and the set $A_t({\bf x})$, see \cite{BNBV2009Forward}. In particular, by choosing a certain bounded set $A_t({\bf x})$, one can easily construct models which induce a covariance structure with bounded support; such models are typically sought after in  applications, which feature a certain  decorrelation time and distance.
Under suitable regularity assumptions on $h$ and on  $A_t({\bf x})$, see \cite{BNBV2011}, the random field defined in (\ref{base}) will be made stationary in time and homogeneous in space. 
It should be noted that in any concrete application, 
one  needs to account for  components in addition to the base model, such as a potential drift, trend and  seasonality, and  observation error on the data level.

One of the key questions we try to answer in this paper is the following one:
How can stochastic volatility/intermittency be introduced in our base model \eqref{base}?  We propose four complementary methods  for tempo-spatial volatility modulation. First, stochastic volatility can be introduced by stochastically changing  the amplitude of the \Levy basis $L$. This can be achieved by adding a stochastic integrand to the base model. This method has frequently been  used in the purely temporal case to account for volatility clusters. In the tempo-spatial case, one needs to establish a suitable stochastic integration theory, which allows for stochastic integrands in the form of random fields. Moreover, suitable models for tempo-spatial stochastic volatility fields need to be developed. 

In the purely temporal, examples typically used to model stochastic volatility are  e.g.~constant elasticity of variance processes, in particular, square root diffusions, see \cite{Cox1975,CIR1985}, Ornstein-Uhlenbeck (OU) processes, see \cite{OU1930} and more recently \cite{BNS2001a}, and supOU processes, see \cite{BN2000,BNStelzer2010b}. 
Second, stochastic volatility can be introduced by extended subordination of the \Levy basis. This concept can be viewed as an extension of the concept of stochastic time change as developed by \cite{Bochner1949}, see also \cite{VeraartWinkel2010} for further references,  to a tempo-spatial framework. Last, volatility modulation can be achieved by randomising a volatility/intensity parameter of the \Levy basis $L$. Here we will study both probability mixing and the new concept of \Levy mixing, which has recently been developed by \cite{BNPAT2012}. We will show how these two mixing concepts can be used to account for stochastic volatility/intermittency. 

Altogether, this paper  contributes to the area of  \emph{ambit stochastics}, which is  the name for the theory and application of \emph{ambit fields} and \emph{ambit processes}.
Ambit stochastics is a  new field of mathematical stochastics that has its origin in the study of turbulence, see e.g.~\cite{BNSch04}, but is in fact of broad applicability in science, technology and finance, in relation to modelling of spatio-temporal dynamic processes. E.g.~important applications of ambit stochastics include modelling turbulence in physics, see e.g.~\cite{BNSch04,BNSch09, Hedevang2012},  modelling tumour growth in biology, see \cite{BNSch07a,JSVJ2008}, and applications in financial mathematics, see \cite{BNBV2010a, VV2,BNBV2009Forward}. 

The outline for the remainder of this article is as follows.
Section \ref{Section:Background} reviews the concept of \Levy bases and integration with respect to \Levy bases, where the focus is on the integration theories developed by \cite{RajRos89} and \cite{W}.
Integrals with respect to \Levy bases are then used to establish the notion for our base model \eqref{base} and for general ambit fields and ambit processes in Section \ref{SectionAmbit}. In addition to reviewing the  general framework of ambit fields, we establish new smoothness and semimartingale conditions for ambit fields.
An important sub-class of ambit fields -- the so-called \emph{trawl processes}, which constitute a class of stationary infinitely divisible stochastic processes -- are then presented in 
Section \ref{Section:Subclasses}.  
Section \ref{Section:SpecVol} focuses on  volatility modulation and establishes four complementary concepts which can be used to model tempo-spatial stochastic volatility/intermittency: Stochastic scaling of the amplitude through a  stochastic integrand, time change and extended subordination of a random measure, and  probability and \Levy mixing. Finally, Section \ref{Section:Conclusion}  concludes and gives an outlook on future research.


\section{Background}\label{Section:Background} 
Ambit fields and ambit processes are constructed from so-called \Levy bases. We will now review the definition and key properties of such \Levy bases and then describe how stochastic integrals can be defined with respect to \Levy bases.

\subsection{Background on L\'{e}vy bases}\label{BackgroundLBasis}
Our review is based on the work by
 \cite{RajRos89,P}, where detailed proofs can be found. 

 Throughout the paper, we denote by $(\Omega, \mathcal{F}, P)$ a  probability space.  Also, let
$(S,\mathcal{S}, \leb)$ denote a Lebesgue-Borel space where 
 $S$ denotes a 
 Borel set in 
$\mathbb{R}^k$ for a $k \in \mathbb{N}$, e.g.~often we choose $S= \mathbb{R}^k$. Moreover  $\mathcal{S}=\mathcal{B}(S)$ is the Borel $\sigma$-algebra on $S$ and $\leb$ denotes the Lebesgue measure.
In addition, we define 
$$\mathcal{B}_b(S)= \{A \in \mathcal{S}: \leb(A) < \infty \},$$
which is the subset of $\mathcal{S}$ that contains sets which have bounded Lebesgue measure. 
Note that since $\leb$ is  $\sigma$-finite, we can deduce that $\mathcal{S}=\sigma(\mathcal{B}_b(S))$, see \citet[p.~399]{PeccatiTaqqu2008}. 
Also, the set $\mathcal{B}_b(S)$ is  
closed under finite union, relative complementation, and countable intersection and is therefore a $\delta$-ring. 



\subsubsection{Random measures}
Random measures play a key role in  ambit stochastics, hence we start off by recalling 
 the definition of a (full) random measure.
\begin{defi}  
\begin{enumerate}
\item By a \emph{random measure} $M$ on $(S,\mathcal{S})$ 
we mean a collection of $\R$-valued random variables   
$\left\{ M\left(A\right): A\in 
\mathcal{B}_b(S)\right\} $ such that for any
sequence $A_{1},A_{2},\dots$ of disjoint elements of 
$\mathcal{B}_b(S)$ 
 satisfying $\cup _{j=1}^{\infty }A_j\in \mathcal{B}_b(S)$
we have $M\left( \cup _{j=1}^{\infty }A_j\right) =\sum_{j=1}^{\infty }M\left(
A_{j}\right) $ a.s..
\item  By a \emph{full random measure} $M$ on $(S,\mathcal{S})$   
\ we mean a random object whose realisations are measures on $(S,\mathcal{S})$ 
a.s..
\end{enumerate}
\end{defi}
Note here that a realisation of a random measure is in general not an ordinary signed measure since it does not necessarily have finite variation. That is why we also introduced the term of a full random measure. Other articles or textbooks would sometimes call the quantity we have defined as a random measure as \emph{random noise} to stress that it might not be a (signed) measure, see \citet[p.~118]{STaqqu1994} for a discussion of this aspect.

In some applications, we work  with stationary random measures, which are defined as follows.
\begin{defi}
 A (full) random measure on
$\mathcal{S}$ 
 is said to be \emph{stationary} if for any ${\bf s}\in S$ and any
finite collection $A_{1},A_{2},\dots,A_{n}$ of elements (of $\mathcal{B}(%
S)$) of $\mathcal{B}_{b}(\mathcal{S})$ the random vector $$\left(
M\left( A_{1}+{\bf s}\right) ,M\left( A_{2}+{\bf s}\right) ,\dots,M\left( A_{n}+{\bf s}\right)
\right) $$ has the same law as $\left( M\left( A_{1}\right) ,M\left(
A_{2}\right) ,\dots,M\left( A_{n}\right) \right) $.
\end{defi}
The above definition ensures that a random measure is stationary in all components. One could also study stationarity in the individual components separately.

\subsubsection{\Levy bases}
In this paper,  we work with a special class of random measures, called \emph{\Levy bases}. 
Before we can define them, we define independently scattered random measures.
\begin{defi}
A \emph{random measure} $M$ on $(S,\mathcal{S})$ 
is \emph{independently scattered} if 
 for any
sequence $A_{1},A_{2},\dots$ of disjoint elements of 
$\mathcal{B}_b(S)$, the random variables  
$M(A_{1}), M(A_{2}),\dots$ are independent.
\end{defi}

Recall the definition of infinite divisibility of a distribution.
\begin{defi}
The law $\mu$ of a random variable on $\R$ is \emph{infinitely divisible} (ID) if for any $n \in \N$, there exists a law $\mu_n$ on $\R$ such that $\mu = \mu_n^{*n}$, where $\mu_n^{*n}$ denotes the $n$-fold convolution of $\mu_n$ with itself.
\end{defi}
In the following, we are interested in ID random measures, which we define now.
\begin{defi} 
 A random measure $M$\ on $(S,\mathcal{S})$\ is said to
be infinitely divisible if for any finite collection $A_{1},A_{2},\dots,A_{n}$
of elements of $\mathcal{B}_b(S)$ the random vector $\left(
M\left( A_{1}\right) ,M\left( A_{2}\right) ,\dots,M\left( A_{n}\right) \right) 
$ is infinitely divisible.
\end{defi}
Let us study one relevant example.
\begin{ex} Assume that $M$\ is an absolutely continuous full
random measure on $(S,\mathcal{S})$ with a density $m $\ and suppose that the
stochastic process $\left\{ m \left( x\right) \right\} _{x \in S} $\ is non-negative and infinitely divisible. Then $M$\ is an infinitely
divisible full random measure.
\end{ex}
Now we can give the definition of a \Levy basis.
\begin{defi}
\begin{enumerate}
\item 
A \emph{L\'{e}vy basis} $L$ on $(S,\mathcal{S})$ is an
 independently scattered, infinitely divisible random measure.
\item 
A \emph{homogeneous L\'{e}vy basis} on $(S,\mathcal{S})$
is a stationary, 
 independently scattered, infinitely divisible random measure.
\end{enumerate}
\end{defi}

\subsubsection{L\'{e}vy-Khintchine formula and
L\'{e}vy-It\^{o} decomposition
}
Since a \Levy basis $L$ is ID, it has 
a \levy-Khintchine representation. I.e.~let $\zeta \in \R$ and $A \in \mathcal{B}_b(S)$, then  
\begin{align}\begin{split}\label{LevKhi}
C\{\zeta \ddagger L(A)\}&=\log\left(\mathbb{E}(\exp(i \zeta L(A))\right)\\
& = i \zeta a^*(A) - \frac{1}{2}\zeta^2 b^*(A) + \int_{\mathbb{R}}\left(e^{i\zeta x}-1-i\zeta x \mathbb{I}_{[-1,1]}(x)\right) n(dx, A), 
\end{split}\end{align}
where, according to  \citet[Proposition 2.1 (a)]{RajRos89},  
$a^*$ is a signed measure on $\mathcal{B}_b(S)$, $b^*$ is a measure on $\mathcal{B}_b(S)$, and 
$n(\cdot,\cdot)$ is the generalised L\'{e}vy measure, i.e.~$n(dx, A)$ is a L\'{e}vy measure on $\mathbb{R}$ for fixed $A \in \mathcal{B}_b(S)$ and a measure on $\mathcal{B}_b(S)$ for fixed $dx$.

Next, we define the \emph{control measure} as introduced in \citet[Proposition 2.1 (c), Definition 2.2]{RajRos89}.
\begin{defi} 
Let $L$ be a \Levy basis with \levy-Khintchine representation (\ref{LevKhi}). 
Then, the measure 
 $c$ is  defined   by
\begin{align}\label{ControlM}
c(B) = |a^*|(B) + b^*(B) + \int_{\R}\min(1,x^2)n(dx,B),\qquad B \in \mathcal{S}_b,
\end{align}
 where $|\cdot|$ denotes the total variation.
 The extension of the measure $c$ to a $\sigma$-finite measure on 
$(S, \mathcal{S})$
 is called the \emph{control measure} measure of $L$.
 \end{defi} 
 Based on the control measure we can now characterise the generalised \Levy measure $n$ further, see \citet[Lemma 2.3, Proposition 2.4]{RajRos89}.  
 First of all, we define the Radon-Nikodym derivatives of the three components of $c$, which are given by
 \begin{align}\label{CQ}
 a({\bf z}) = \frac{da^*}{dc}({\bf z}), &&
 b({\bf z}) = \frac{db^*}{dc}({\bf z}), &&
 \nu(dx, {\bf z}) = \frac{n(dx,\cdot)}{dc}({\bf z}).
 \end{align}
 Hence, we have in particular that
 $n(dx, d{\bf z}) = \nu(dx, {\bf z}) c(d{\bf z})$.
Without loss of generality we can assume  that $\nu(dx, {\bf z})$ is a \Levy measure for each fixed ${\bf z}$ and hence we do so in the following.

\begin{defi}\label{CondCQ}
We call $(a, b, \nu(dx, \cdot), c)=(a({\bf z}), b({\bf z}), \nu(dx, {\bf z}), c(d{\bf z}))_{{\bf z}\in S}$ a  \emph{characteristic quadruplet} (CQ) associated with a \Levy basis $L$ on 
$(S, \mathcal{S})$
 provided the following conditions hold:
\begin{enumerate}
\item 
Both $a$ and $b$ are functions on 
$S$, 
 where $b$ is restricted to be non-negative. 
\item  For fixed ${\bf z}$, $\nu(dx,{\bf z})$ is  a \Levy measure on $\R$, and, for fixed $dx$, it is a measurable function on $S$. 
\item The control measure $c$ is a measure on
$(S, \mathcal{S})$
 such that $\int_B a({\bf z})c(d{\bf z})$ is a (possibly signed) measure on 
 $(S, \mathcal{S})$,
 $\int_B b({\bf z})c(d{\bf z})$ is a  measure on
$(S, \mathcal{S})$
and 
$\int_B\nu(dx,{\bf z})c(d{\bf z})$ is a \Levy measure on $\R$ for fixed 
$B \in \mathcal{S}$.
\end{enumerate}
\end{defi}
We have seen that every \Levy basis on $(S, \mathcal{S})$ determines a CQ of the form 
$(a, b, \nu(dx, \cdot), c)=(a({\bf z}), b({\bf z}), \nu(dx, {\bf z}), c(d{\bf z}))_{{\bf z}\in S}$.
And, conversely, every CQ satisfying the conditions in Definition \ref{CondCQ} determines, in law, a \Levy basis on $(S, \mathcal{S})$ .

In a next step, we relate the notion of \Levy bases and CQs to the concept of Poisson random measures and their compensators.

\begin{defi}
A \Levy basis on $(S, \mathcal{S})$  is \emph{dispersive} if its control measure $c$ satisfies $c(\{{\bf z}\})=0$ for all ${\bf z}\in S$.
\end{defi}
For a dispersive L\'{e}vy basis $L$  on $(S, \mathcal{S})$ with characteristic quadruplet $(a, b, \nu, c)$  there is a modification $L^*$ with the same characteristic quadruplet which has following L\'{e}vy-It\^{o} decomposition:
\begin{align}\label{LevyIto}
L^*(A) 
&= a^*(A)+ W(A) +
\int_{\{|y|\leq 1\}} y(N- n)(dy,A)
+ \int_{\{|y|> 1\}} y N(dy,A), 
\end{align}
for
$A\in \mathcal{B}_b(S)$ and  
for a Gaussian basis $W$ (with characteristic quadruplet $(0, b, 0, c)$, i.e. $W(A)\sim N(0,\int_Ab({\bf z})c(d{\bf z}))$), and a Poisson basis $N$ (independent of $W$) with compensator
$%
n\left( dy;A\right) =\mathrm{E}\left\{ N\left( dy;A\right) \right\}$ where $n(dx,d{\bf z})= \nu(dx, {\bf z}) c(d{\bf z})$, cf.~\cite{P} and \citet[Theorem 2.2]{BNStelzer2010b}

It is also possible to write (\ref{LevyIto}) in infinitesimal form by
\begin{align}\label{LIi}
L^*(d{\bf z}) = a^*(d{\bf z}) + W(d{\bf z}) + \int_{\{|x|>1\}}x N(dx, d{\bf z}) + \int_{\{|x|\leq 1\}} x(N-n)(dx, d{\bf z}).
\end{align}
This is particularly useful in the context of the \levy-Khintchine representation, which can then also be expressed in infinitesimal form by
\begin{align}\begin{split}\label{LevKhii}
C \{\zeta \ddagger L(d{\bf z})\}&=\log\left(\mathbb{E}(\exp(i \zeta L(d{\bf z}))\right)\\
& = i \zeta a^*(d{\bf z}) - \frac{1}{2}\zeta^2 b^*(d{\bf z}) + \int_{\mathbb{R}}\left(e^{i\zeta x}-1-i\zeta x \mathbb{I}_{[-1,1]}(x)\right) n(dx, d{\bf z})\\
& = \left(i \zeta a({\bf z}) - \frac{1}{2}\zeta^2 b({\bf z}) + \int_{\mathbb{R}}\left(e^{i\zeta x}-1-i\zeta x \mathbb{I}_{[-1,1]}(x)\right) \nu(dx, {\bf z})\right)c(d{\bf z})\\
&= C\{\zeta \ddagger L'({\bf z})\} c(d{\bf z}), \qquad \zeta \in \R,
\end{split}\end{align}
where $L'({\bf z})$ denotes the \emph{Levy seed} of $L$ at ${\bf z}$.  Note that  $L'({\bf z})$ is defined as the infinitely divisible random variable having \levy-Khintchine representation 
\begin{align}\label{CharL'}
 C\{\zeta \ddagger L'({\bf z})\} = i \zeta a({\bf z}) - \frac{1}{2}\zeta^2 b({\bf z}) + \int_{\mathbb{R}}\left(e^{i\zeta x}-1-i\zeta x \mathbb{I}_{[-1,1]}(x)\right) \nu(dx, {\bf z}).
\end{align}

\begin{rem}
We can associate a \Levy process with any \Levy seed. In particular, let $L'({\bf z})$ denote the \Levy seed of $L$ at ${\bf z}$. Then, $(L_t'({\bf z}))_t$ denotes the \Levy process generated by $L'({\bf z})$, which is defined as the \Levy process whose law is determined by $L_1'({\bf z}) \stackrel{law}{=} L'({\bf z})$.
\end{rem}

\begin{defi}
Let $L$ denote a \Levy basis on $(S, \mathcal{S})$ with CQ given by $(a,b, \nu(dx,\cdot),c)$.
\begin{enumerate}
\item 
If $\nu(dr, {\bf z})$ does not depend on ${\bf z}$, we call  $L$ \emph{factorisable}. 
\item
If $L$ is factorisable and if $c$ is proportional to the Lebesgue measure and $a({\bf z})$ and $b({\bf z})$ do not depend on $z$, then $L$ is called \emph{homogeneous}. In that case we write $c(d{\bf z}) = v \leb(d{\bf z}) = v d{\bf z}$ for a positive constant $v>0$ and where $\leb(\cdot)$ denotes the Lebesgue measure.
\end{enumerate}
\end{defi} 
In order to simplify the exposition, we will throughout this paper assume that in the case of a homogeneous \Levy basis the constant $v$ is set to 1, i.e.~the measure $c$ is given by the Lebesgue measure.

\subsubsection{Examples of \Levy bases}
Let us study some examples of \Levy bases $L$ on $(\R^k, \mathcal{B}(\R^k))$  with CQ $(a, b, \nu(dx, \cdot), c)$.
\begin{ex}[Gaussian \Levy basis]
When 
 $\nu(dx,{\bf z}) \equiv 0$, then $L$ constitutes a
Gaussian \Levy basis with  $L(A)\sim\mathit{N}\left(\int_A a({\bf z})c(d{\bf z}), \int_A b({\bf z})c(d{\bf z})\right)$, for $A\in \mathcal{B}_b(\R^k)$. If, in addition, $L$ is homogeneous, then $L(A)\sim\mathit{N}\left(a \leb(A), b \leb(A)\right)$.
\end{ex}

\begin{ex}[Poisson \Levy  basis]
When $c(d{\bf z}) = d{\bf z}$ and $a\equiv b \equiv 0$ and $\nu(dx;{\bf z})= \lambda({\bf z}) \delta_1(dx)$, where $\delta_1$ denotes the Dirac measure with point mass at 1 and $\lambda({\bf z}) > 0$ is the intensity function, then $L$ constitutes a  Poisson \Levy basis. If, in addition, $L$ is factorisable, i.e.~$\lambda$ does not depend on ${\bf z}$, then 
   $L(A)\sim \mathit{Poisson}\left(\lambda\leb(A)\right)$, for all $A\in \mathcal{B}_b(\R^k)$.
\end{ex}

\begin{ex}[Gamma \Levy basis] Suppose that $c(d{\bf z}) = d{\bf z}$, $a \equiv b \equiv 0$ and the (generalised) \Levy measure is of the form
$\nu(dx; {\bf z}) = x^{-1}e^{-\alpha({\bf z})x}dx$,
where $\alpha({\bf z})> 0$. In that case, we call the corresponding \Levy basis $L$ a \emph{gamma \Levy basis}. If, in addition, $L$ is factorisable, i.e.~the function $\alpha$ does not depend on the parameter ${\bf z}$, then $L(A)$ has a gamma law for all $A\in \mathcal{B}_b(\R^k)$.
\end{ex}

\begin{ex}[Inverse Gaussian \Levy basis] Suppose that $c(d{\bf z}) = d{\bf z}$, $a \equiv b \equiv 0$ and the (generalised) \Levy measure is of the form
$\nu(dx; {\bf z}) = x^{-3/2}e^{-\frac{1}{2}\gamma^2({\bf z})x}dx$,
where $\gamma({\bf z})> 0$. Then we call the corresponding \Levy basis $L$ an \emph{inverse Gaussian \Levy basis}. If, in addition, $L$ is factorisable, i.e.~the function $\gamma$ does not depend on the parameter ${\bf z}$, then $L(A)$ has an inverse Gaussian law for all $A\in \mathcal{B}_b(\R^k)$.
\end{ex}

\begin{ex}[\Levy process]
If $k=1$, i.e.~$L$ is a \Levy basis on $\R$, then $L([0,t])=L_t$, $t\geq 0$ is a \Levy process.
\end{ex}

\subsection{Integration concepts with respect to a L\'{e}vy basis }\label{SectWalsh}
In order to build  relevant models based on \Levy bases, we need a suitable integration theory. In the following, we will briefly review the integration theory developed by \cite{RajRos89} and also the one by \cite{W}, and we refer  to \cite{BNBV2011} for a detailed overview on integration concepts with respect to \Levy bases, see also   \cite{DalangQuer2011} for a related review and \cite{BGP2012} for details on integration with respect to multiparameter processes with stationary increments.

\subsubsection{The integration concept  by \cite{RajRos89}}
According to \citet[p.460]{RajRos89}, integration of suitable deterministic functions with respect to \Levy bases can be defined as follows.
First define an integral for simple functions:
\begin{defi}\label{RajRosSimpleInt}
Let $L$ be a \Levy basis on $(S, \mathcal{S})$. 
Define a simple function on $S$, i.e.~let $f:=\sum_{j=1}^nx_j \mathbb{I}_{A_j}$, where  $A_j \in \mathcal{B}_b(S)$, for $j=1, \dots,n$,  are disjoint.
Then, one defines the integral, for every $A \in \sigma(\mathcal{B}_b(S))=\mathcal{S}$, by
\begin{align*}
\int_A f dL := \sum_{j=1}^n x_j L(A\cap A_j).
\end{align*}
\end{defi}
The integral for general  
 measurable functions can be derived by a limit argument.
 \begin{defi}
 Let $L$ be a \Levy basis on $(S, \mathcal{S})$.  A measurable function $f:(S, \mathcal{S}) \mapsto (\R, \mathcal{B}(\R))$ is called $L$-\emph{measurable} if there exists a sequence $\{f_n\}$ of simple function  as in Definition \ref{RajRosSimpleInt}, such that
 \begin{itemize}
 \item $f_n \to f$ a.e.~(i.e. $c$-a.e., where $c$ is the control measure of $L$),
 \item for every $A \in \mathcal{S}$, the sequence of simple integrals $\{\int_A f_n dL\}$ converges in probability, as $n \to \infty$. 
 \end{itemize}
 For integrable measurable functions,  define
 \begin{align*}
 \int_A f dL := \mathbb{P}-\lim_{n \to \infty} \int_A f_n dL.
 \end{align*}
  \end{defi}
 \cite{RajRos89} have pointed out that the above integral is well-defined in the sense that it does not depend on the approximating sequence $\{f_n\}$.
 Also, the necessary and sufficient conditions for the existence of the integral $\int f dL$ can be expressed in terms of the characteristics of $L$ and can be found in \citet[Theorem 2.7]{RajRos89}, which says the following. Let  $f:(S, \mathcal{S}) \mapsto (\R, \mathcal{B}(\R))$ be a measurable function.
 Let $L$ be a \Levy basis with CQ $(a, b, \nu(dx, \cdot), c)$.
Then $f$ is integrable w.r.t.\  $L$ if and only if the following three conditions are satisfied:
\begin{align}\label{RajRosIntCond}
\int_S | V_1(f({\bf z}),{\bf z})|c(d{\bf z}) < \infty,&&
 \int_S |f({\bf z})|^2  b({\bf z})c(d{\bf z}) < \infty, &&
\int_S V_2(f({\bf z}),{\bf z})c(d{\bf z}) < \infty,
\end{align}
 where for $\varrho(x) := x\indicator_{[-1, 1]}(x)$,
\begin{align*}
V_1(u,{\bf z}):= u a({\bf z}) +\int_{\mathbb{R}}\left(\varrho(xu) -u\varrho(x)\right)\nu(dx, {\bf z}),&&
V_2(u,{\bf z}) := \int_{\mathbb{R}} \min(1,|xu|^2)\nu(dx, {\bf z}).
\end{align*}

 Note that such integrals have been defined for \emph{deterministic} integrands. However, in the context of ambit fields, which we will focus on in this paper, we typically encounter  stochastic integrands representing  stochastic volatility, which  tends to be present in most applications we have in mind. 
 Since we often work under the independence assumption that the stochastic volatility $\sigma$ and the \Levy basis $L$ are independent, it has been suggested to work with conditioning to extend the definition by \cite{RajRos89} to allow for stochastic integrands.
 An alternative concept, which directly allows for stochastic integrands which can be dependent of the \Levy basis, is the integration concept by  \cite{W}, which we study next.

\subsubsection{Integration w.r.t.~martingale measures introduced by \cite{W}}
The integration theory due to \cite{W} can be regarded as It\^{o} integration extended to random fields.
 In the following we will present the integration theory on a bounded domain and comment later on how one can extend the theory to the case of an unbounded domain.
 
 Here we  treat \emph{time} and \emph{space} separately, which allows us to work with a natural ordering (introduced by time) and to relate the integrals w.r.t.~to \Levy bases to martingale measures.
 In the following, we denote by $S$  a \emph{bounded} Borel set in $\mathcal{X}=\mathbb{R}^d$ for a $d \in \mathbb{N}_0$ (where $d+1 = k$) and  $\mathcal{S}=\mathcal{B}(S)$ denotes the Borel $\sigma$-algebra on $S$. Since $S$ is bounded, we have in fact $\mathcal{S}=\mathcal{B}(S)=\mathcal{B}_b(S)$.
 
Let $L$ denote a \Levy basis on $(S\times  [0,T], \mathcal{B}(S\times  [0,T]))$  for some $T>0$. 
For any  $A \in \mathcal{B}_b(S)$ and $0 \leq t \leq T$, we define
\begin{align*}
L_t(A) = L(A,t) = L( A \times (0,t]).
\end{align*}
Here $L_t(\cdot)$ is a measure-valued process, and  for a fixed set $A \in \mathcal{B}_b(S)$, $L_t(A)$ is an additive process in law.
In the following, we want to use the $L_t(A)$ as integrators as in \cite{W}. In order to do that, we work under the square-integrability assumption, i.e.:

\begin{description}
\item[Assumption (A1):]
For each $A \in \mathcal{B}_b(S)$, we have that $L_t(A) \in L^2(\Omega, \mathcal{F},  P)$.
\end{description} 
In the following, we will, unless otherwise stated, work without loss of generality under the zero-mean assumption on $L$, i.e.\
\begin{description}
\item[Assumption (A2):]
For each $A \in \mathcal{B}_b(S)$, we have that $\mathbb{E}(L_t(A))=0$.
\end{description} 

Next,  we define the filtration  $\mathcal{F}_t$ by 
\begin{align}\label{filt}
\mathcal{F}_t = \cap_{n=1}^{\infty} \mathcal{F}_{t+1/n}^0,&& \text{where}\qquad  \mathcal{F}_t^0 =\sigma \{L_s(A): A \in \mathcal{B}_b(S), 0 < s \leq t \}\vee \mathcal{N},
\end{align}
and 
where $\mathcal{N}$ denotes the $P$-null sets of $\mathcal{F}$.
Note that  $\mathcal{F}_t$ is right-continuous by construction.
One can show that under the assumptions (A1) and (A2) and for fixed $A \in \mathcal{B}_b(S)$, $(L_t(A))_{0\leq t \leq T}$ is a (square-integrable) \emph{martingale} with respect to the filtration $(\mathcal{F}_t)_{0\leq t \leq T}$. 
Note that these two properties together with 
the fact that $L_0(A)=0$ a.s.\ 
ensure that $(L_t(A))_{t \geq 0, A \in \mathcal{B}(R^d)}$ is a \emph{martingale measure} with respect to $(\mathcal{F}_t)_{0\leq t \leq T}$ in the sense of \cite{W}. 
Furthermore, we have the following  orthogonality 
property: If $A, B \in \mathcal{B}_b(S)$ with $A \cap B = \emptyset$, then $L_t(A)$ and $L_t(B)$ are independent. 
Martingale measures which satisfy such an orthogonality property are referred to as 
 \emph{orthogonal martingale measures} by  \cite{W}, see also \cite{BNBV2011} for more details. Note that orthogonal martingale measure are \emph{worthy}, see \citet[Corollary 2.9]{W}, a property which makes  them suitable as integrators.
For such orthogonal martingale measures, \cite{W} introduces their   \emph{covariance measure}  $Q$ by
\begin{align}\label{CovMeasureQ}
Q(A \times [0,t]) = \ <L(A)>_t,
\end{align} 
for $A \in \mathcal{B}(S)$.
Note that  $Q$ is a positive measure and is used by \cite{W}   when defining stochastic integration with respect to $L$. 

\cite{W} defines stochastic integration in the following way.
Let $\zeta(\xi,s)$ be an \emph{elementary}
 random field $\zeta(\xi,s)$, i.e.\ it has the form
\begin{equation}\label{elrv}
\zeta({\boldsymbol \xi},s,\omega)=X(\omega)\indicator_{(a,b]}(s)\indicator_A({\boldsymbol \xi})\,,
\end{equation} 
where $0\leq a<t$, $a\leq b$, $X$ is bounded and $\mathcal{F}_a$-measurable, and $A\in\mathcal{S}$. For such elementary
functions, the stochastic integral with respect to $L$ can be defined as
\begin{equation}
\int_0^t\int_{B}\zeta({\boldsymbol \xi},s)\,L(d{\boldsymbol \xi},ds):=X\left(L_{t\wedge b}(A\cap B)-L_{t\wedge a}(A\cap B)\right)\,,
\end{equation}
for every $B\in\mathcal{S}$. 
It turns out that  the stochastic integral becomes a martingale measure itself in $B$ (for fixed $a, b, A$). 
Clearly, the above integral can easily be generalised to allow for integrands given by {\it simple} random fields, i.e.\ 
finite linear combinations of elementary random fields. 
Let $\mathcal{T}$ denote the
 set of simple random fields and let the {\it predictable} $\sigma$-algebra $\mathcal{P}$ be the 
$\sigma$-algebra generated by $\mathcal{T}$. Then we call  a random field  {\it predictable} provided 
it is $\mathcal{P}$-measurable. 
The aim is now to define stochastic integrals with respect to $L$ where the integrand is given by a predictable random field.

In order to do that \cite{W} defines a norm $\|\cdot\|_L$ on the predictable random fields $\zeta$ by
\begin{equation}\label{norm}
\|\zeta\|_L^2:=\E\left[\int_{[0,T]\times S}\zeta^2({\boldsymbol \xi},s)\,Q(d{\boldsymbol \xi}, ds)\right]\,,
\end{equation}
which determines  the Hilbert space $\mathcal{P}_L:=L^2(\Omega\times[0,T]\times S,\mathcal{P}, Q)$, which is the space of predictable random fields $\zeta$ with    $\|\zeta\|_L^2< \infty$, and he  shows that 
$\mathcal{T}$ is dense in $\mathcal{P}_L$. 
Hence, in order to define the stochastic integral of $\zeta\in\mathcal{P}_L$, one can  choose an approximating
sequence $\{\zeta_n\}_n\subset\mathcal{T}$ such that $\|\zeta-\zeta_n\|_L\rightarrow 0$ as $n\rightarrow\infty$. Clearly, for each $A\in\mathcal{S}$, $\int_{[0,t]\times A}\zeta_n({\boldsymbol \xi},s)\,L(d{\boldsymbol \xi},ds)$ is a Cauchy sequence
in $L^2(\Omega,\mathcal{F},P)$, and thus there exists a limit which is defined as the stochastic integral
of $\zeta$. 

Then, this stochastic integral  is again a martingale measure and satisfies the following \emph{It\^{o}-type isometry}:
\begin{equation}\label{ItoIs}
\E\left[\left(\int_{[0,T]\times S}\zeta({\boldsymbol \xi},s)L(d{\boldsymbol \xi}, ds)\right)^2\right]=\|\zeta\|_L^2\, ,
\end{equation}
see \cite[Theorem 2.5]{W} for more details. 

\subsubsection{Relation between the two integration concepts}
The relation between the two different integration concept has been discussed in \citet[pp.~60--61]{BNBV2011}, hence we only 
mention it briefly here.

Note that the \cite{W} theory defines the stochastic integral as the $L^2$-limit of simple random fields, whereas \cite{RajRos89} work with the  $P$-limit. \cite{BNBV2011} point out  that \emph{deterministic} integrands, which are integrable  in the sense of Walsh, are thus also integrable in the Rajput and Rosinski sense since the control measure of \cite{RajRos89} and the covariance measure of \cite{W} are equivalent.


\section{General aspects of the theory of ambit fields and processes}\label{SectionAmbit}
 In the following we will show how stochastic processes and random fields can be constructed based  on \Levy bases, which leads us to the general framework of ambit fields.
This section reviews the concept of ambit fields and ambit processes. For a detailed account on this topic see \cite{BNBV2011} and \cite{BNSch07a}.

 \subsection{The general framework}\label{SubsectAmPr}
The general framework for defining an ambit process is as follows. Let
  $Y=\left\{ Y_{t}\left( {\bf x}\right) \right\} $  with $Y_t({\bf x}):= Y({\bf x},t)$  denoting a stochastic field in
space-time $\mathcal{X}\times \mathbb{R}$.
In most applications, the space $\mathcal{X}$ is  chosen to be $\mathbb{R}^{d}$ for $d=1,2$ or $3$.
 Let $\varpi \left( \theta
\right) =\left({\bf x}\left( \theta \right),t\left( \theta \right) \right) $ denote a curve in $\mathcal{X}\times \mathbb{R}$.
The  values of the field along the
curve  are then given by $X_{\theta }=Y_{t\left( \theta
\right) }\left( {\bf x}\left( \theta \right) \right)$.
 Clearly, $X=\left\{ X_{\theta }\right\}$ denotes a stochastic process.  Further, 
 the stochastic field is assumed to be generated by innovations in space-time
with values $Y_{t}\left( {\bf x} \right)$ which are supposed to depend only on
innovations that occur prior to or at time $t$ and in general only on a restricted set of the corresponding part of space-time. I.e., at each
point $\left({\bf x}, t\right)$,  the value of $Y_{t}\left( {\bf x}\right)$  is only determined 
by innovations in some subset $A_{t}\left(
{\bf x}\right) $\ of $\mathcal{X}\times \mathbb{R}_{t}$\ (where $\mathbb{R}%
_{t}=(-\infty ,t]$), which we call the \emph{ambit set} associated to $\left({\bf x}, t\right)$. 
Furthermore,  
 we  refer to $Y$ and $X$\ as an \emph{ambit
field}\ and an \emph{ambit process}, respectively. 


 In order to use such general ambit fields in applications, we have to impose some structural assumptions. More precisely, we will define
 $Y_{t}\left( {\bf x}\right) $
as a stochastic integral plus a drift term,
where the integrand in the stochastic integral will consist of a deterministic
kernel times a positive random variate which is taken to express the \emph{%
volatility} of the field $Y$. 
More precisely,  we think of ambit fields as being defined as follows.
\begin{defi} Using the notation introduced above, an \emph{ambit field} is defined as a random field of the form
\begin{align}\label{ambfi}
Y_{t}\left( {\bf x}\right) =\mu +\int_{A_{t}\left( {\bf x}\right) }h\left({\bf x}, t; {\boldsymbol \xi}, s\right) \sigma _{s}\left( {\boldsymbol \xi} \right) L\left(  \mathrm{d}{\boldsymbol \xi}, \mathrm{d}s \right) 
 +\int_{D_{t}\left( {\bf x}\right) }q\left( {\bf x},t; {\boldsymbol \xi} ,s\right) a_{s}\left( {\boldsymbol \xi}
\right)   \mathrm{d}{\boldsymbol \xi} \mathrm{d}s,  
\end{align}
provided the integrals exist,
where $A_{t}\left( {\bf x}\right) $, and $D_{t}\left({\bf  x}\right) $ are ambit sets, $%
h $\ and $q$\ are deterministic functions, $\sigma \geq 0$ is a stochastic
field referred to as  \emph{volatility} or \emph{intermittency}, $a$ is also a stochastic field, and $L$\
is a \emph{L\'{e}vy basis}.
\begin{rem}
Note that compared to the base model \eqref{base} we introduced in the Introduction, the ambit field defined in \eqref{ambfi} also comes with a drift term and stochastic volatility introduced in form of a stochastic integrand. In Section \ref{Section:SpecVol}, we will describe in detail how such a stochastic volatility field $\sigma$ can be specified and what kind of complementary routes can be taken in order to allow for stochastic volatility clustering.
\end{rem}
The corresponding ambit process $X$ along the curve $\varpi$ is then given by
\begin{align}\label{ambitprocess}
X_{\theta} = \mu + \int_{A(\theta)}h(t(\theta); {\boldsymbol \xi},s) \sigma_s({\boldsymbol \xi}) L(d{\boldsymbol \xi},ds) + 
\int_{D(\theta)} q(t(\theta); {\boldsymbol \xi},s) a_s({\boldsymbol \xi})d{\boldsymbol \xi} ds,
\end{align}
where $A(\theta) = A_{t(\theta)}({\bf x}(\theta))$ and 
$D(\theta) = D_{t(\theta)}({\bf x}(\theta))$.
\end{defi}
In Section \ref{IntAmbitFields}, we will formulate the suitable integrability conditions which guarantee the existence of the integrals above.

Of particular interest in many applications are ambit processes that are stationary in time and
nonanticipative and homogeneous in space. More specifically, they may be derived from ambit fields $Y$
of the form%
\begin{align}
Y_{t}\left({\bf x}\right) &=\mu +\int_{A_{t}\left({\bf x}\right) }h\left( {\bf x}-{\boldsymbol \xi}, t-s\right) \sigma_{s}\left( {\boldsymbol \xi} \right) L\left( \mathrm{d}{\boldsymbol \xi} ,\mathrm{d}%
s\right)  +\int_{D_{t}\left({\bf x}\right) }q\left( {\bf x}-{\boldsymbol \xi}, t-s\right) a_{s}\left( {\boldsymbol \xi}
\right) \mathrm{d}{\boldsymbol \xi} \mathrm{d}s.  \label{ambfistat}
\end{align}%
Here the ambit sets $A_{t}\left( {\bf x}\right) $ and $D_{t}\left( {\bf x}\right) $ are
taken to be \emph{homogeneous} and \emph{nonanticipative},  i.e. $A_{t}\left(
{\bf x}\right) $\ is of the form $A_{t}\left( {\bf x}\right) =A+\left( {\bf x},t\right) $\
where $A$\ only involves negative time coordinates, and similarly for $%
D_{t}\left( {\bf x}\right) $. In addition, $\sigma$ and $a$ are chosen to be stationary in time and space and $L$ to be  homogeneous.



\subsection{Integration for general ambit fields}\label{IntAmbitFields}
Ambit fields have initially been  defined for deterministic integrands using the \cite{RajRos89} integration concept. Their definition could then be extended  to allow for stochastic integrands which are independent of the \Levy basis by a conditioning argument. 
As discussed before, 
 the integration framework developed by \cite{W}  has the advantage that it allows 
for stochastic integrands which are potentially dependent of the \Levy basis and enables us 
to study dynamic properties (such as martingale properties).
 Let us explain in more detail  how the \cite{W} integration concept can be used to define ambit fields using an It\^o-type integration concept.


One concern regarding the applicability of the \cite{W} framework to ambit fields might be that 
  general ambit sets $A_t(x)$ are not necessarily bounded, and we have only presented the \cite{W} concept for a bounded domain. 
However, the stochastic integration concept reviewed above can be extended to unbounded ambit sets using   
standard arguments, cf.\ 
\citet[p.~289]{W}.
Also, as pointed out in \citet[p.~292]{W}, it is possible to extend the \cite{W} integration concept  beyond the $L^2$-framework, cf.~\citet[p.~292]{W}.

Note that the classical \cite{W} framework works under the zero mean assumption, which might not be satisfied for general ambit fields. However, we can always define a new \Levy basis $\overline L$ by setting
$\overline L := L -\mathbb{E}(L)$,
which clearly has zero mean. Then we can define the \cite{W} integral w.r.t.~$\overline L$, and we obtain an additional drift term which needs to satisfy an additional integrability condition.

However, the main point we need to address is the fact that the integrand in  the ambit field does not seem to comply with the structure of the integrand in the Walsh-theory.
More precisely, for
  ambit fields with ambit sets $A_t({\bf x}) \subset \R^d \times (-\infty,t]$, we would like to define Walsh-type integrals for integrands of the form 
\begin{align}\label{Walshintegrand}
\zeta({\boldsymbol \xi},s):= \zeta({\bf x}, t; {\boldsymbol \xi},s) :=  \indicator_{A_t({\bf x})}({\boldsymbol \xi},s) h({\bf x}, t; {\boldsymbol \xi}, s) \sigma_s({\boldsymbol \xi}).
\end{align}
 The original Walsh's integration theory covers integrands which do not depend on the time index $t$. 
Clearly, the integrand given in (\ref{Walshintegrand}) generally exhibits $t$-dependence due to the choice of the ambit set $A_t({\bf x})$ and due to the deterministic kernel function $h$.

Suppose we are in the  simple case where  the ambit set can be represented as $A_t({\bf x}) = B \times (-\infty,t]$, where $B\in\mathcal{B}(\R^d)$ does not depend on $t$, and where the kernel function does not depend on $t$, i.e. $h({\bf x}, t; {\boldsymbol \xi}, s) = h({\bf x}; {\boldsymbol \xi}, s)$. Then  the Walsh-theory is directly applicable, and provided the integrand is indeed Walsh-integrable, then  (for fixed $B$ and fixed ${\bf x}$) the process
$$\left(\int_{-\infty}^t \int_{B} h({\bf x}; {\boldsymbol \xi}, s)\sigma_s({\boldsymbol \xi}) L(d{\boldsymbol \xi}, ds)\right)_{t\in \R}$$ is a martingale.

Note that the $t$-dependence (and also the additional ${\bf x}$-dependence) for general integrands in the ambit field is in  the deterministic part of the integrand only, i.e.~in $\indicator_{A_t({\bf x})}({\boldsymbol \xi},s) h({\bf x}, t; {\boldsymbol \xi}, s)$.
Now in order to allow for time $t$- (and ${\bf x}$-) dependence in the integrand, we can define the integrals in the Walsh sense for any \emph{fixed} $t$ and for \emph{fixed} ${\bf x}$.
Note that we treat ${\bf x}$ as an additional parameter which does not have an influence on the structural properties of the integral as a stochastic process in $t$.

It is clear that  in the case of having $t$-dependence in the integrand, the resulting stochastic integral is, in general, not a martingale measure any more.
However,  the properties of adaptedness, square-integrability and countable additivity carry over to the process
\begin{align*}
\left(\int_{-\infty}^t \int_{\R^d} \zeta({\bf x}, t; {\boldsymbol \xi}, s) L(d{\boldsymbol \xi}, ds)\right)_{t\in \R}
\end{align*}
(for fixed ${\bf x}$) 
since it is the $L^2$-limit of a stochastic process with the above mentioned properties.

 In order to ensure that the ambit fields (as defined in (\ref{ambfi})) are well-defined (in the Walsh-sense), throughout the rest of the paper we will work under the following assumption: 
 \begin{as}\label{as:Walsh-intcond}
Let $L$ denote a \Levy basis on
$S \times (-\infty, T]$, where $S$ denotes a not necessarily bounded Borel set $S$ in $\mathcal{X}=\mathbb{R}^d$ for some $d \in \N$. Define the new \Levy basis $\overline L := L -\mathbb{E}(L)$.
We extend the definition of the covariance measure  $Q$ of $\overline L$, see (\ref{CovMeasureQ}),  to an unbounded domain  and,  next, we define a Hilbert space $\mathcal{P}_{\overline L}$ with norm $||\cdot||_{\overline L}$ as in (\ref{norm}) (extended to an unbounded domain) and, hence, we have an \Ito isometry of type (\ref{ItoIs}) extended to an unbounded domain.
We assume that, for fixed ${\bf x}$ and  $t$,
\begin{align*}
\zeta({\boldsymbol \xi},s) =  \indicator_{A_t({\bf x})}({\boldsymbol \xi},s) h({\bf x}, t; {\boldsymbol \xi}, s) \sigma_s({\boldsymbol \xi})
\end{align*}
satisfies
\begin{enumerate}
\item $\zeta\in \mathcal{P}_{\overline L}$,
\item $||\zeta||_{\overline L^2} = \E\left[ \int_{\mathbb{R}\times \mathcal{X}}\zeta^2({\boldsymbol \xi}, s)Q(d{\boldsymbol \xi},ds)\right] < \infty$.
\item  $\int_{\mathbb{R}\times \mathcal{X}}|\zeta({\boldsymbol \xi}, s)|\mathbb{E}L(d{\boldsymbol \xi},ds) < \infty$
\end{enumerate}
 
 \end{as}

\begin{rem}
Note that alternatively, we could work with  the c\`{a}dl\`{a}g elementary random fields 
\begin{align*}
\zeta^*({\boldsymbol \xi}, s, \omega):=X(\omega)\indicator_{[a, b)}(s)\indicator_A({\boldsymbol \xi}),
\end{align*}
where $X(\omega)$ is assumed to be $\mathcal{F}_a$-adapted and the remaining notation is as in (\ref{elrv}).
Next, one can construct a $\sigma$-algebra from  the corresponding simple random fields and one would then define  the stochastic integral for $\zeta^*({\boldsymbol \xi}, s-, \omega)$, since clearly adaptedness and the
 c\`{a}dl\`{a}g property of $\zeta^*({\boldsymbol \xi}, s, \omega)$ implies predictability of $\zeta^*({\boldsymbol \xi}, s-, \omega)$.
\end{rem}


\subsection{Cumulant function for stochastic integrals w.r.t.~a \Levy basis}
Next we study some of the fundamental properties of ambit fields. 
Throughout this subsection, we work under the following assumption:
\begin{as}\label{as:indep}
The stochastic fields $\sigma$ and $a$ are independent of the \Levy basis $L$.
\end{as}
Now we have all the tools at hand which are  needed to compute the conditional characteristic function of ambit fields defined in (\ref{ambfi}) where $\sigma$ and $L$ are assumed independent and where we condition on the $\sigma$-algebra $\mathcal{F}^{\sigma}=\mathcal{F}_t^{\sigma}({\bf x})$ which is generated by the history of $\sigma$, i.e.
\begin{align*}
\mathcal{F}^{\sigma}_t({\bf x}) = \sigma\{\sigma_s({\boldsymbol \xi}): ({\boldsymbol \xi}, s)\in A_t({\bf x})\}.
\end{align*}
 \begin{prop} \label{prop:cumulant-ambit}
Assume that Assumption \ref{as:indep} holds.
Let
$C^{\sigma}$
 denote the conditional cumulant function 
when we condition on the volatility field $\sigma$. The conditional cumulant function of the ambit field defined by (\ref{ambfi})
is given by 
 \begin{align}\begin{split}\label{CharFct}
&C^{\sigma}\left\{\theta\ddagger \int_{A_t({\bf x})} h({\bf x},t; {\boldsymbol \xi},s)\sigma_s({\boldsymbol \xi})L(d{\boldsymbol \xi},ds)\right\} \\
&= 
\log\left(\E\left.\left(\exp\left(i \theta \int_{A_t({\bf x})}h({\bf x},t; {\boldsymbol \xi},s)\sigma_s({\boldsymbol \xi})L(d{\boldsymbol \xi},ds)\right)\right|\mathcal{F}^{\sigma}\right)\right)
 \\
&= \int_{A_t({\bf x})}C\left\{ \theta h({\bf x},t;{\boldsymbol \xi},s)\sigma_s({\boldsymbol \xi})\ddagger {L'({\boldsymbol \xi}, s)}\right \}c(d{\boldsymbol \xi}, ds),
 \end{split}\end{align}
 where $L'$ denotes the \Levy seed 
and
 $c$ is the control measure   associated with the \Levy basis $L$, cf.\ (\ref{CharL'}) and (\ref{ControlM}).
 \end{prop}
 \begin{proof}
 The proof of the Proposition is an immediate consequence of 
  \citet[Proposition 2.6]{RajRos89}.
\end{proof}
 \begin{cor}
In the case where $L$ is a homogeneous \Levy basis, equation (\ref{CharFct}) simplifies to 
 \begin{align*}
C^{\sigma}\left\{\theta \ddagger \int_{A_t({\bf x})} h({\bf x},t; {\boldsymbol \xi},s)\sigma_s({\boldsymbol \xi})L(d{\boldsymbol \xi},ds)\right\}
= \int_{A_t({\bf x})}C\left\{ \theta h({\bf x},t;{\boldsymbol \xi},s)\sigma_s({\boldsymbol \xi})\ddagger {L'}\right \}d{\boldsymbol \xi} ds.
 \end{align*}
\end{cor}

\subsection{Second order structure of ambit fields} \label{SecondOrderStr}
Next we study the second order structure of ambit fields.
Throughout the Section, let
\begin{align}\label{AmbitFieldWithoutDrift}
Y_t({\bf x}) = \int_{A_t({\bf x})}h({\bf x},t;{\boldsymbol \xi},s)\sigma_s({\boldsymbol \xi})L(d{\boldsymbol \xi}, ds),
\end{align}
where $\sigma$ is independent of $L$, i.e.~Assumption \ref{as:indep} holds, and $L'$ is the \Levy seed associated with $L$.

 \begin{prop}\label{CorStructureAmbit}
Let $t, \tilde t, {\bf x},{\bf \tilde x} \geq 0$ and let $Y_t({\bf x})$ be an ambit field as defined in (\ref{AmbitFieldWithoutDrift}).
The second order structure is then as follows. The means are given by
\begin{align*}
\mathbb{E}\left.\left(Y_t({\bf x})\right|\mathcal{F}^{\sigma}\right)
&= 
 \int_{A_t({\bf x})}h({\bf x}, t; {\boldsymbol \xi}, s)\sigma_s({\boldsymbol \xi})\E(L'({\boldsymbol \xi}, s))c(d{\boldsymbol \xi}, ds),\\
\mathbb{E}\left(Y_t({\bf x})\right)
& = 
 \int_{A_t({\bf x})}h({\bf x}, t; {\boldsymbol \xi}, s)\mathbb{E}\left(\sigma_s({\boldsymbol \xi})\right)\E(L'({\boldsymbol \xi}, s)) c(d{\boldsymbol \xi}, ds).
\end{align*}
The variances are given by
\begin{align*}
Var\left.\left(Y_t({\bf x})\right|\mathcal{F}^{\sigma}\right) 
&= 
\int_{A_t({\bf x})}h^2({\bf x}, t; {\boldsymbol \xi},s)\sigma_s^2({\boldsymbol \xi})Var(L'({\boldsymbol \xi}, s))c(d{\boldsymbol \xi}, ds),\\
 Var\left(Y_t({\bf x})\right) 
 &= 
 \int_{A_t({\bf x})}h^2({\bf x}, t; {\boldsymbol \xi},s)\mathbb{E}\left(\sigma_s^2({\boldsymbol \xi})\right)Var(L'({\boldsymbol \xi}, s))c(d{\boldsymbol \xi}, ds)\\
 &\hspace{-1.5cm}+  \int_{A_t({\bf x})}\int_{A_t({\bf x})}h({\bf x}, t; {\boldsymbol \xi},s)h({\bf x}, t; \tilde {\boldsymbol  \xi},\tilde s)
 \rho(s,\tilde s, {\boldsymbol \xi}, \tilde {\boldsymbol \xi})\E(L'({\boldsymbol \xi}, s))
 \E(L'(\tilde {\boldsymbol \xi}, \tilde s))
 c( d{\boldsymbol \xi}, d s) c( d\tilde {\boldsymbol \xi}, d \tilde s),
\end{align*}
where $\rho(s,\tilde s, {\boldsymbol \xi}, \tilde {\boldsymbol \xi}) = 
\mathbb{E}\left(\sigma_s({\boldsymbol \xi})\sigma_{\tilde s}(\tilde {\boldsymbol \xi})\right)
-
\mathbb{E}\left(\sigma_s({\boldsymbol \xi})\right)
\mathbb{E}\left(\sigma_{\tilde s}(\tilde {\boldsymbol \xi})\right)$.
The covariances are given by
\begin{align*}
Cov\left.\left(Y_t({\bf x}), Y_{\tilde t}(\tilde {\bf x})\right| \mathcal{F}^{\sigma} \right) &=
\int_{A_t({\bf x})\cap A_{\tilde t}(\tilde {\bf x})}h({\bf x}, t; {\boldsymbol \xi},s)h(\tilde {\bf x}, \tilde t; {\boldsymbol \xi},s)\sigma_s^2({\boldsymbol \xi}) Var(L'({\boldsymbol \xi}, s))c(d{\boldsymbol \xi}, ds),\\
Cov\left(Y_t({\bf x}), Y_{\tilde t}(\tilde {\bf x}) \right) &=
 \int_{A_t({\bf x})\cap A_{\tilde t}(\tilde {\bf x})}h({\bf x}, t; {\boldsymbol \xi},s)h(\tilde {\bf x}, \tilde t; {\boldsymbol \xi},s)\mathbb{E}\left(\sigma_s^2({\boldsymbol \xi})\right) Var(L'({\boldsymbol \xi}, s))c(d{\boldsymbol \xi}, ds)
\\
&\hspace{-2cm}+ \int_{A_t({\bf x})}\int_{A_{\tilde t}(\tilde {\bf  x})} 
h({\bf x}, t; {\boldsymbol \xi},s)h(\tilde {\bf x}, \tilde t; \tilde {\boldsymbol \xi},\tilde s)
\rho(s,\tilde s,{\boldsymbol \xi},\tilde {\boldsymbol \xi}) \E(L'({\boldsymbol \xi}, s))\E(L'(\tilde {\boldsymbol \xi}, \tilde s))c(d \tilde {\boldsymbol \xi}, d\tilde s)c(d{\boldsymbol \xi}, ds).
\end{align*}

\end{prop}
\begin{proof}
From the conditional cumulant function \eqref{CharFct}, we can easily deduce the second order structure conditional on the stochastic volatility. Integrating over $\sigma$ and using the law of total variance and covariance leads to the corresponding unconditional results.  
\end{proof}
Note that it is straightforward to generalise the above results to allow for an additional drift term as in \eqref{ambfi}.

The second order structure provides us with some valuable insight into the autocorrelation structure of an ambit field. Knowledge of the autocorrelation structure can help us to study smoothness properties of an ambit field, as we do in the following section.
Also, from a more practical point of view, we could think of specifying a fully parametric model based on the ambit field. Then  the second order structure could  be used e.g.~in a (quasi-) maximum-likelihood set-up to estimate the model parameters.

\subsection{Smoothness conditions}
Let us study sufficient conditions which ensure smoothness of an ambit field.

\subsubsection{Some related results in the literature}
In the purely temporal (or null-spatial) case, which we will discuss in more detail in Section \ref{VMLV}, smoothness conditions for so-called Volterra processes have been studied before. In particular, 
 \cite{Decreusefond2002} shows that under mild integrability assumptions on a
progressively measurable  stochastic volatility process, the sample-paths of the volatility modulated Brownian-driven Volterra process
are a.s. H\"{o}lder-continuous even for some singular deterministic kernels. Note that \cite{Decreusefond2002} does not use the term \emph{stochastic volatility} in his article, but the stochastic integrand he considers could be regarded as a stochastic volatility process. 
Also,
 \cite{MytnikNeuman2011} study sample path properties of Volterra processes.

In the tempo-spatial context, or generally for random fields, smoothness conditions have been discussed in detail in the literature. For textbook treatments see e.g.~\cite{Adler1981},   \cite{AdlerTaylor2007} and \cite{AzaisWschebor2009}.
Important articles in this context include the following ones.
\cite{Kent1989} formulates sufficient conditions on the covariance function of a stationary real-valued random field which ensure sample path continuity.

 \cite{Rosinski1989} studies the relationship between the sample-path properties of an infinitely divisible integral process and the properties of the sections of the deterministic kernel. The study is carried out under the assumption of absence of a Gaussian component. In particular, he shows that various properties of the section are inherited by the paths of the process, which include boundedness, continuity, differentiability, integrability and boundedness of $p$th variation. 

\cite{MarcusRosinski2005} extend the previous results to derive sufficient conditions for boundedness and continuity for  stochastically continuous infinitely divisible processes, without Gaussian component.


\subsubsection{Sufficient condition on the covariance function}
In the following, we write
\begin{align*}
\rho(t,{\bf x}; \tilde t, \tilde {\bf x})= Cov\left(Y_t({\bf x}), Y_{\tilde t}(\tilde {\bf x}) \right),
\end{align*}
where the covariance is given as in Proposition \ref{CorStructureAmbit}.
We can apply the key results derived in \cite{Kent1989} to ambit fields.

Let $t, h_1 \in \R$ and ${\bf x}, {\bf h_2} \in \R^d$ and ${\bf h}:=(h_1, {\bf h_2})'$. For each $({\bf x},t)'$ we assume that $\rho(t+h_1,{\bf x}+{\bf h_2};  t-h_1, {\bf x}-{\bf h_2})$ is $k$-times continuously differentiable with respect to ${\bf h}$ for a $k \in \N_0$. We write
$p_k({\bf h}; t, {\bf x})$ for the polynomial of degree $k$ which is obtained from a Taylor expansion of $\rho(t+h_1,{\bf x}+{\bf h_2};  t-h_1, {\bf x}-{\bf h_2})$ about ${\bf h}={\bf 0}$ for each $({\bf x}, t)$. In the following, we denote by $||\cdot||$ the Euclidean norm.

\begin{prop}
For each $({\bf x},t)'$ suppose that $\rho(t+h_1,{\bf x}+{\bf h_2};  t-h_1, {\bf x}-{\bf h_2})$ is $d+1$-times continuously differentiable with respect to ${\bf h}$ and that there exists a constant $\gamma > 0$ such that 
\begin{align}\label{SmoothC1}
\sup_{({\bf x},t)} \left\{\left| \rho (t+h_1,{\bf x}+{\bf h_2};  t-h_1, {\bf x}-{\bf h_2})- p_{d+1}({\bf h}; t, {\bf x})\right| \right\} = O\left(\frac{||{\bf h}||^{d+1}}{|\log ||{\bf h}|||^{3 + \gamma}}\right), 
\end{align}
as $||{\bf h}|| \to 0$,
where the supremum is computed over $({\bf x},t)$ being in each compact subset of $\R^{d+1}$. 
Then there exists a version of the  random field $\{Y_t({\bf x}),({\bf x},t)\in \R^{d+1}\}$ which has almost surely continuous sample paths.
\end{prop}
\begin{proof}
The result is a  direct consequence of \citet[Theorem 1 and Remark 6]{Kent1989}.
\end{proof}

\begin{rem}
As pointed out in 
\citet[Remark 3]{Kent1989},  \eqref{SmoothC1} could be replaced by the stronger conditions of the supremum being of order $O(||{\bf h}||^{d+1+\beta})$ as ${\bf h}\to {\bf 0}$ for some constant $\beta > 0$. This condition is easier to check in practice and if it holds for any $\beta > 0$, then it implies that \eqref{SmoothC1} holds for all $\gamma > 0$.
\end{rem}

\begin{rem}
As pointed out in 
\citet[Remark 5]{Kent1989} and \citet[p.~60]{Adler1981}, as soon as we have a Gaussian field, milder conditions ensure continuity. Note that an ambit field is Gaussian if the \Levy basis is Gaussian and  the stochastic volatility component is absent (or purely deterministic).
\end{rem}


\subsection{Semimartingale conditions}
Next we derive sufficient conditions which ensure that an ambit field is a semimartingale in time. This is interesting since in financial applications we typically want to stay within the semimartingale framework whereas in applications to turbulence one typically focuses on non-semimartingales.

We will  see that a sufficient  condition for a semimartingale  is linked to a smoothness condition on the kernel function. 
When studying the semimartingale condition, we focus on ambit sets which factorise as $A_t({\bf x}) = [0, t] \times A({\bf x})$, which is in line with the \cite{W}-framework.
We start with a preliminary Lemma.
\begin{lem}
Let $L$ be a \Levy basis satisfying (A1) and (A2) and $\sigma$ be a 
predictable stochastic volatility field which is integrable w.r.t.~$L$. Then 
\begin{align}\label{M}
M_t(A({\bf x}))=\int_0^t\int_{A({\bf x})}\sigma_s({\boldsymbol \xi})L(d{\boldsymbol \xi}, ds),
\end{align}
 is an orthogonal martingale measure.
\end{lem}
\begin{proof}
See \citet[p.~296]{W} for a proof of the lemma above.
\end{proof}

\begin{as}\label{as:smcond}
Assume that the deterministic function $u\mapsto h(\cdot,u, \cdot, \cdot)$ is differentiable (in the second component) and denote by $h'=\frac{\partial h}{\partial u}(\cdot,u, \cdot, \cdot)$  the derivative with respect to the second component. Then for  $s\leq t$ we have the representation 
\begin{align}\label{hAss}
h({\bf x}, t; {\boldsymbol \xi}, s) = h({\bf x}, s;  {\boldsymbol \xi}, s) + \int_s^t h'({\bf x}, u; {\boldsymbol \xi}, s) du,
\end{align}
provided that  $h({\bf x}, s; {\boldsymbol \xi}, s)$ exists.
Further, we assume that both components in the representation \eqref{hAss} are \cite{W}-integrable w.r.t.~$M$ and 
$\zeta_1({\boldsymbol \xi}, s)=h({\bf x}, s;  {\boldsymbol \xi}, s)$ and $\zeta_2({\boldsymbol \xi}, s)= \int_s^t h'({\bf x}, u; {\boldsymbol \xi}, s) du$ satisfy 
$\zeta_i \in \mathcal{P}_M$, $||\zeta_i||_{M^2}< \infty$, 
$\int_{R \times \mathcal{X}} |\zeta_i({\boldsymbol \xi},s)|\E M(d{\boldsymbol \xi},ds)< \infty$ for $i=1, 2$.
\end{as}

\begin{prop}
Let $M$ be defined as in \eqref{M} with covariance measure $Q_M$,  
 and assume that $h$ satisfies Assumption \ref{as:smcond} and 
\begin{align*}
\mathbb{E} \int_{S \times S \times [0,T] \times G}
|F(u,{\bf x}; s, {\boldsymbol \xi_1})
F(u,{\bf x}; s, {\boldsymbol\xi_2})| Q_M(d{\boldsymbol\xi_1}, d{\boldsymbol\xi_2}, ds) du < \infty,
\end{align*}
for
\begin{align*}
F(u, {\bf x}; s, {\boldsymbol\xi})=h'({\bf x}, u;{\boldsymbol\xi}, s)\mathbb{I}_{[s,t]}(u).
\end{align*}

Then $(\int_0^t \int_{A({\bf x})} h({\bf x}, t; {\boldsymbol\xi}, s) M(d{\boldsymbol\xi}, ds))_{t\geq 0}$ is a semimartingale with representation 
\begin{multline}\label{semimg}
\int_0^t \int_{A({\bf x})} h(t, {\bf x}; s, {\boldsymbol\xi}) M(d{\boldsymbol\xi}, ds)
\\
= \int_0^t \int_{A({\bf x})} h(s, {\bf x}; s, {\boldsymbol\xi}) M(d{\boldsymbol\xi}, ds)
+\int_0^t \int_0^u \int_{A({\bf x})}  h'(u, {\bf x}; s, {\boldsymbol\xi}) M(d{\boldsymbol\xi}, ds) du.
\end{multline}
\end{prop}
Clearly, the first term in  representation \eqref{semimg} is a martingale measure in the sense of Walsh and the second term is a finite variation process.
\begin{proof}
The result follows from the stochastic  Fubini theorem for martingale measures, see \citet[Theorem 2.6]{W}.
In the following, we check that the conditions of the stochastic Fubini theorem are satisfied.
Let $(G, \mathcal{G}, \leb)$ denote a finite measure space. Concretely, take $G=[0,T]$ and  $\mathcal{G} =\mathcal{B}(G)$. 
Note that $u\mapsto h(\cdot, u, \cdot,\cdot)$ has finite first derivative at all points, hence its derivative is $\mathcal{G}$-measurable. Also, the indicator function is $\mathcal{G}$-measurable since the corresponding interval is an element of $\mathcal{G}$. Overall we have that the function
\begin{align*}
F(u, {\bf x}; s, {\boldsymbol\xi})=h'({\bf x}, u;{\boldsymbol\xi}, s)\mathbb{I}_{[s,t]}(u)
\end{align*}
is $\mathcal{G}\times \mathcal{P}$-measurable and 
\begin{align*}
\mathbb{E} \int_{S \times S \times [0,T] \times G}
|F(u,{\bf x}; s,{\boldsymbol \xi_1})
F(u,{\bf x}; s, {\boldsymbol\xi_2})| Q_M(d{\boldsymbol\xi_1}, d{\boldsymbol\xi_2}, ds) du < \infty,
\end{align*}
where $Q_M$ is the covariance measure of $M$.
Then
\begin{align*}
&\int_0^t \int_{A({\bf x})} h(t,{\bf x}; s, {\boldsymbol\xi}) M(d{\boldsymbol\xi}, ds)
\\
&= \int_0^t \int_{A({\bf x})} h(s, {\bf x}; s, {\boldsymbol\xi}) M(d{\boldsymbol\xi}, ds)
+ \int_0^t \int_{A({\bf x})} \int_s^t h'(u, {\bf x}; s, {\boldsymbol\xi})du M(d{\boldsymbol\xi}, ds)\\
&= \int_0^t \int_{A({\bf x})} h(s, {\bf x}; s, {\boldsymbol\xi}) M(d{\boldsymbol\xi}, ds)
+\int_0^t \int_0^u \int_{A({\bf x})}  h'(u, {\bf x}; s, {\boldsymbol\xi}) M(d{\boldsymbol\xi}, ds) du.
\end{align*}
\end{proof}

\subsection{The purely-temporal case: Volatility modulated Volterra processes}\label{VMLV}
The purely temporal, i.e.~the null-spatial, case of an ambit field has been studied in detail in recent years.  
Here we denote by
  $L=(L_t)_{t \in \mathbb{R}}$  a  L\'{e}vy process on $\mathbb{R}$. Then, the null--spatial ambit field is in fact a \emph{volatility modulated \levy-driven Volterra} ($\mathcal{VMLV}$) process  denoted by 
$Y=\left\{ Y_{t}\right\} _{t\in \mathbb{R}}$, where
\begin{equation}
Y_{t}=\mu +\int_{-\infty }^{t}k(t,s)\sigma _{s-}\mathrm{d}L_{s}+\int_{-\infty
}^{t}q(t,s)a_{s}\mathrm{d}s,  \label{VMV}
\end{equation}%
where $\mu $ is a constant, $k$ and $q$ are
are real--valued measurable function on $\R^2$, 
such that the  integrals above  exist
 with $k\left( t,s\right)=q\left( t,s\right) =0$ for $s>t$, and $\sigma$ and $a$ are c\`{a}dl\`{a}g  processes.

Of particular interest, are typically 
semi--stationary  processes, i.e.\ the case when  the kernel function depends on $t$ and $s$ only through the difference $t-s$.  This determines  the class of 
 L\'{e}vy semistationary processes ($\mathcal{LSS}$), see \cite{BNBV2010a}.
  Specifically, 
\begin{equation}
Y_{t}=\mu +\int_{-\infty }^{t}k(t-s)\sigma _{s-}\mathrm{d}Z_{s}+\int_{-\infty
}^{t}q(t-s)a_{s}\mathrm{d}s,  \label{LSS}
\end{equation}%
$k$ and $q$ are
non-negative deterministic functions on $\mathbb{R}$ with $k\left( x\right)=q\left(x\right) =0$ for $x\leq 0$.
Note that an $\mathcal{LSS}$ process is stationary as soon as $\sigma$ and $a$ are stationary processes. 
In the case that $L$ is a Brownian motion, we call $Y$ a \emph{Brownian semistationary} ($\mathcal{BSS}$) \emph{process}, see \cite{BNSch09}.

The class of $\mathcal{BSS}$ processes has been used by \cite{BNSch09} to model turbulence in physics. In that context, intermittency, which is modelled by $\sigma$, plays a key role, which has triggered detailed research on the question of how intermittency can be estimated non-parametrically. Recent research, see e.g.~\cite{BNCP2009a,BNCP09,BNCP10b}, has developed realised multipower variation and related concepts to tackle this important question. 

The class of $\mathcal{LSS}$ processes has subsequently been found to be suitable for modelling energy spot prices, see \cite{BNBV2010a,VV2}.
Moreover, \cite{BNBPV2012} have recently developed an anticipative  stochastic integration theory with respect to $\mathcal{VMLV}$ processes.

\section{Illustrative example: Trawl processes}\label{Section:Subclasses}
Ambit fields and processes constitute a very flexible class for modelling a variety of tempo-spatial phenomena. This Section will focus on one particular   sub-class of ambit processes, which has recently been used  in application in turbulence and finance.

\subsection{Definition and general properties}

 \emph{Trawl processes} are stochastic processes defined in terms of tempo-spatial \Levy bases. They  have  recently been introduced in \cite{BN2011} as a class of stationary infinitely divisible stochastic processes.
\begin{defi}\label{DefTrawlProc}
Let $L$ be a homogeneous \Levy basis on $\R^d \times \R$ for $d \in \N$. Then, using the same notation as before, 
\begin{align*}
\mathcal{B}_{b}(\R^d \times \R) := \{ A \in \mathcal{B}(\R^d \times \R): \leb(A) < \infty \}.
\end{align*}
 Further, for an $A=A_0 \in \mathcal{B}_{b}(\R^d \times \R)$, we define $A_t:= A +({\bf 0},t)$.
Then
\begin{align}\label{trawlproc}
Y_t = \int_{\R^d \times \R} 1_A({\boldsymbol\xi}, s-t)L(d{\boldsymbol\xi}, ds) =L(A_t),
\end{align}
defines the \emph{trawl process} associated with the \Levy basis $L$ and the \emph{trawl} $A$. 
\end{defi}
The assumption $\leb(A)< \infty$ in the definition above ensures the existence of the integral in (\ref{trawlproc}) as defined by \cite{RajRos89}. 
\begin{rem} 
 If $A \subset (-\infty, 0]\times \R^{d}$, then the trawl process belongs to the class of ambit processes.
\end{rem}

The intuition and also the name of the trawl process comes from the idea that 
we have a certain tempo-spatial set -- the trawl -- which is relevant for our object of interest. 
  I.e.~the object of interest at time $t$ is modelled as the \Levy basis evaluated over the trawl $A_t$. As time $t$ progresses,
we pull along the trawl (like a fishing net) and hence obtain a stochastic process (in time $t$). For the time being, we have in mind that  the shape of the trawl does not change as time progresses, i.e.~that the process is stationary. This assumption can be relaxed as we will discuss in Section \ref{Section:SpecVol}.

\begin{ex}
Let 
 $d=1$ and suppose 
that the trawl is given by $A_t =\{(x, s): s\leq t, 0 \leq x \leq \exp(-0.7(t-s))\}$. 
Figure \ref{Fig-Trawl} illustrates the basic framework for such a process. It depicts the trawl at different times $t \in \{2, 5\}$. The value of the process at time $t$ is then determined by evaluating the corresponding \Levy basis over the trawl $A_t$.


\begin{figure}[htb]
\centering
\includegraphics[trim=4cm 20cm 3cm 3.5cm, clip=true, totalheight=0.3\textheight]{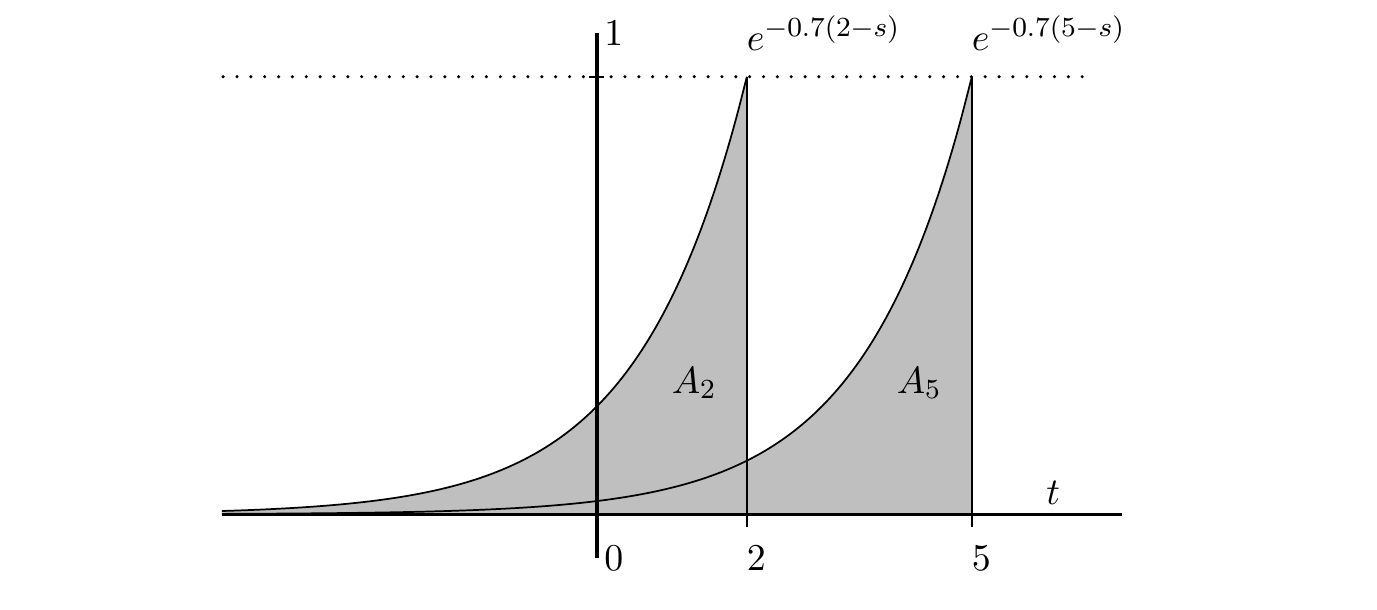}
\caption{Example of a relevant choice of the trawl:
$A_t =\{(x, s): s\leq t, 0 \leq x \leq \exp(-0.7(t-s))\}$. The shape of the trawl does not change as time $t$ progresses and, hence, we obtain a stationary process. The value of the process is obtained by evaluating $L(A_t)$ for each $t$.\label{Fig-Trawl}}
\end{figure}

\end{ex}

\subsection{Cumulants and correlation structure}\label{Cum}
From \cite{BN2011} we know that the trawl process defined above is a strictly stationary stochastic process, and 
 we can easily derive the cumulant transform of a  trawl process,
which is given by
\begin{align}\label{CumTrawl}
C(\zeta \ddagger Y_t):= \E(\exp(i \zeta Y_t))
= \leb(A) C(\zeta \ddagger L') :=\leb(A) \E(\exp(i \zeta L')).
\end{align}	
From the equation (\ref{CumTrawl}), we see immediately that the law of the trawl process is infinitely divisible since the corresponding \Levy seed has infinitely divisible law.

\begin{rem} 
 To any infinitely divisible law  $\pi$ there exists a stationary trawl process having $\pi$ as its one-dimensional marginal law. This follows from formula \eqref{CumTrawl}, cf.~\cite{BN2011}
\end{rem}
We can now easily derive the cumulants of the trawl process which are given by
$\kappa_i(Y_t)=\leb(A) \kappa_i(L')$, for $i \in \N$, 
 provided they exist.
In particular, the mean and variance  are given by
\begin{align*}
\E(Y_t) = \leb(A)\  \E(L'), &&
Var(Y_t)=  \leb(A)\ Var(L').
\end{align*}


Marginally, we see that the precise shape of the trawl does not have any impact on the distribution of the process. The quantity which matters here is the \emph{size}, i.e.~the Lebesgue measure,  of the trawl. So two different specifications of the ambit set, which have the same size, are not identified based on  the marginal distribution only.

However, we will see that the shape of the trawl determines the autocorrelation function. More precisely, the autocorrelation structure is given as follows. Let $h > 0$, then 
\begin{align}\label{TrawlAcf}
\rho(h):= Cov(Y_t, Y_{t+h}) =  \leb(A \cap A_h) Var(L').
\end{align}
For the autocorrelation, we get
\begin{align*}
r(h) = Cor(Y_t, Y_{t+h})=  \frac{\leb(A \cap A_h)}{\leb(A)}.
\end{align*}

\subsection{\levy-It\^{o} decomposition for trawl processes}
In the following we study some of the sample path properties of trawl processes.
First of all, we study a representation result for trawl process $Y$, where we split the process into a drift part, a Gaussian part and a jump part in a similar fashion as in the classical L\'{e}vy-It\^{o} decomposition.
Recall that $L$ is a homogeneous \Levy basis on $\R^d\times \R$. 
In the following, $s$ denotes the one dimensional temporal variable and ${\boldsymbol \xi}$ denotes the $d$-dimensional spatial variable.

From the 
\levy-It\^{o} decomposition, see (\ref{LevyIto}), we get the following representation result for the trawl process 
 $Y_t = L(A_t)$  defined in \eqref{trawlproc}, i.e.
 \begin{align}\begin{split}\label{LItrawl}
L(A_t) &= a\leb(A) + W(A_t) + \int_{\{|x|>1\}}x N(dx, A_t) + \int_{\{|x|\leq 1\}} x(N-n)(dx, A_t)\\
&= a \leb(A)+ W(A_t) + \int_{A_t}\int_{\{|x|>1\}}x N(dx, d{\boldsymbol \xi}, ds) + \int_{A_t}\int_{\{|x|\leq 1\}} x(N-n)(dx, d{\boldsymbol \xi}, ds), 
\end{split}\end{align}
where $W(A_t) \sim N(0, b \leb(A))$ and $n(dx, d{\boldsymbol \xi}, ds) =  \nu(dx) d{\boldsymbol \xi} ds$ and   $\leb(A_t) = \leb(A)$.

\begin{ex}
Suppose the trawl process $Y_t$ is defined based on a \Levy basis with characteristic quadruplet $(0, b, 0, \leb)$. Then it can be written as
\begin{align*}
Y_t = W(A_t) \sim N(0, b \leb(A)).
\end{align*}
Assume further that   $d=1$ and that 
the trawl is given by $A_t =\{(x, s): s\leq t, 0 \leq x \leq \exp(-\lambda(t-s))\}$, for a positive constant $\lambda > 0$. Then $\leb(A_t) = \frac{1}{\lambda}$, hence $Y_t = W(A_t) \sim N(0, \frac{b}{\lambda})$ and the autocorrelation function is given by $r(h)= \exp(-\lambda h)$, for $h \geq 0$.
\end{ex}

\subsection{Generalised cumulant functional}
We have already studied the cumulant function of \Levy bases and ambit fields. Here, we will in addition focus on the more general \emph{cumulant functional} of a trawl process, which sheds some light on important properties of  trawl processes.
\begin{defi}
Let $Y=(Y_t)$ denote a stochastic process  and let $\mu$ denote \emph{any} non-random measure such that 
\begin{align*}
\mu(Y)=\int_{\R} Y_t \,\mu(dt) < \infty, 
\end{align*}
where the integral should exist a.s..
The \emph{generalised cumulant functional} of $Y$ w.r.t.~$\mu$ is defined as 
\begin{align*}
C\{\theta \ddagger \mu(Y)\} = \log \E\left(\exp\left(i \theta \mu(Y) \right)\right).
\end{align*}
\end{defi}
When we compute the cumulant functional for a trawl process, we obtain the following result.
 \begin{prop}
 Let $Y_t=L(A_t)$ denote a trawl process and let $\mu$ denote \emph{any} non-random measure such that 
$\mu(Y)=\int_{\R} Y_t \,\mu(dt) < \infty$, a.~s..
Given the trawl $A$, we will further assume that for all ${\boldsymbol \xi} \in \R^{d}$, 
\begin{align*}
h_A({\boldsymbol \xi}, s)=\left(\int_{\R} 1_{A}({\boldsymbol \xi}, s-t)\mu(dt)\right)<\infty,
\end{align*}
and that $h_A({\boldsymbol \xi}, s)$ is integrable with respect to the \Levy basis $L$.
Then the cumulant function of $\mu(Y)$ is given by 
\begin{align}\begin{split}\label{cumfct}
C\{\theta \ddagger \mu(Y)\} 
 &= 
i\theta a\int_{\R\times \R^d}  h_A({\boldsymbol \xi}, s)    d{\boldsymbol \xi}ds
-\frac{1}{2}\theta^2 b\int_{\R\times \R^d}  h_A^2({\boldsymbol \xi}, s)    d{\boldsymbol \xi} ds
\\
&\hspace{0.5cm}+ \int_{\R\times \R^d} \int_{\R} 
 \left( \exp (i\theta u x) -1 -i\theta ux \mathbf{I}_{[-1,1]}(x)\right)\nu(dx)  \chi(du),
\end{split}\end{align}
where $\chi $\ is the measure on $\mathbb{R}$\ obtained by lifting the Lebesgue
measure on $\mathbb{R}^d\times \mathbb{R}$ to $\mathbb{R}$ by the mapping $\left(
{\boldsymbol \xi}, s \right) \rightarrow ({\boldsymbol \xi}, h_A(s, {\boldsymbol \xi}))$.
\end{prop}
\begin{proof}
An application of Fubini's theorem (see e.g.~\cite{BNBasse2009}) yields 
\begin{align*}
\mu(Y)&= \int_{\R} \int_{\R\times \R^d}1_{A}({\boldsymbol \xi}, s-t)L(d{\boldsymbol \xi},ds) \,\mu(dt)
=\int_{\R\times \R^d}\underbrace{\left(\int_{\R} 1_{A}({\boldsymbol \xi}, s-t)\mu(dt)\right)}_{=h_A({\boldsymbol \xi}, s)}L(d{\boldsymbol \xi}^,ds)\\
&= \int_{\R\times \R^d}h_A({\boldsymbol \xi}, s) L(d{\boldsymbol \xi},ds).
\end{align*}
From this we find that the cumulant
function of $\mu(Y)$, i.e.~the generalised  cumulant functional of $Y$ w.r.t.~$\mu$,  is given by
\begin{align*}
&C\{\theta \ddagger \mu(Y)\} = \int_{\R\times \R^d}
C\{\theta h_A({\boldsymbol \xi}, s) \ddagger L'\} d{\boldsymbol \xi}ds\\
 &= 
i\theta a\int_{\R\times \R^d}  h_A({\boldsymbol \xi},s)    d{\boldsymbol \xi}ds
-\frac{1}{2}\theta^2 b\int_{\R\times \R^d}  h_A^2({\boldsymbol \xi},s)   d{\boldsymbol \xi} ds
\\
&\hspace{0.5cm}+ \int_{\R\times \R^d} \int_{\R} 
 \left( \exp (i\theta h_A({\boldsymbol \xi},s) x) -1 -i\theta h_A({\boldsymbol \xi}, s)x \mathbf{I}_{[-1,1]}(x)\right)\nu(dx) d{\boldsymbol \xi} ds.
\end{align*}
Note that the latter part, i.e. the jump part of $C\{\theta \ddagger \mu(Y)\}$, can be recast as 
\begin{align*}
&\int_{\R\times \R^d} \int_{\R} 
 \left( \exp (i\theta h_A({\boldsymbol \xi}, s) x) -1 -i\theta h_A({\boldsymbol \xi}, s)x \mathbf{I}_{[-1,1]}(x)\right)\nu(dx)   d{\boldsymbol \xi} ds\\
 &=\int_{\R\times \R^d} \int_{\R} 
 \left( \exp (i\theta u x) -1 -i\theta ux \mathbf{I}_{[-1,1]}(x)\right)\nu(dx)  \chi(du),
\end{align*}
where $\chi $\ is the measure defined as above. Then the result follows.
\end{proof}

In  the following, we mention three relevant choices of the measure $\mu$. 
Let $\delta_t$ denote the Dirac measure at $t$.
We start with a very simple case:
\begin{ex}
Suppose $\mu(t)= \delta(dt)$ for a fixed $t$. Hence 
\begin{align*}
h_A({\boldsymbol \xi}, s)=\left(\int_{\R} 1_{A}({\boldsymbol \xi}, s-t)\mu(dt)\right)= 1_{A}({\boldsymbol \xi}, s-t)<\infty,
\end{align*}
and clearly, $h_A( {\boldsymbol \xi}, s)$ is integrable with respect to $L$. 
Then
\begin{align*}
C\{\theta \ddagger \mu(Y)\}
=C\{\theta \ddagger Y_{t}\} 
= \leb(A) C\{\theta \ddagger L'\}. 
\end{align*}
\end{ex}
This is exactly the result we derived in Section \ref{Cum} above. 
More interesting is the case when $\mu$ is given by a linear combination of different Dirac measures, since this allows us to derive the joint finite dimensional laws of the trawl process and not just the distribution for fixed $t$.
\begin{ex}
Suppose  $\mu \left( \mathrm{d}t\right) =\theta _{1}\delta
_{t_{1}}\left( \mathrm{d}t\right) +\cdots +\theta _{n}\delta _{t_{n}}\left( 
\mathrm{d}t\right) $ for constants $\theta_1, \dots, \theta_n \in \R$ and times $t_1 < \cdots < t_n$ for $n \in \N$. 
 As in the example before, the integrability conditions which were needed to derive (\ref{cumfct}) are satisfied and, hence, (\ref{cumfct}) gives us the cumulant function of the joint
law of $Y_{t_{1}},\dots,Y_{t_{n}}$.  
\end{ex}
Finally, another case of interest is the integrated trawl process, which we study in the next example.
\begin{ex} Let 
$\mu \left( \mathrm{d}t\right)
=1_{I}\left( t\right) \mathrm{d}t$ for an interval $I$\ of $\mathbb{R}$. Then (\ref{cumfct}) determines  the law of $\int_{I}Y_{s}\mathrm{d}s$.
\end{ex}
\begin{rem}
The last example is particularly relevant if the trawl process is for instance used for modelling stochastic volatility. Note that such an application is feasible since the trawl process is stationary and we can formulate assumptions which would ensure the positivity of the process as well (e.g.~if we work with a \Levy subordinator as the corresponding \Levy seed). In that context, integrated volatility is a quantity of key interest. 
\end{rem}
\subsection{The increment process}\label{SectIncr}
Finally, we focus on the increments of a trawl process.
Note that whatever the type of trawl, we have
the following representation for the increments of the process for $s<t$,%
\begin{equation}
Y_{t}-Y_{s}=L\left( A_{t}\backslash A_{s}\right) -L\left( A_{s}\backslash
A_{t}\right),   \label{Y-Y} \qquad \text{almost surely}.
\end{equation}%
Due to the independence of $L\left( A_{t}\backslash A_{s}\right)$ and $-L\left( A_{s}\backslash A_{t}\right)$, we get the following representation for the cumulant function of the returns
\begin{align}\begin{split}\label{CumIncr}
C(\zeta \ddagger Y_t-Y_s) &= C(\zeta \ddagger L\left( A_{t}\backslash A_{s}\right) )
+ C(-\zeta \ddagger L\left( A_{s}\backslash A_{t}\right) )\\
&= \leb(A_t\backslash A_s)C(\zeta \ddagger L') + \leb(A_s\backslash A_t)C(-\zeta \ddagger L')\\
&= \leb(A_{t-s}\backslash A_0)C(\zeta \ddagger L') + \leb(A_0 \backslash A_{t-s})C(-\zeta \ddagger L').
\end{split}
\end{align}

\subsection{Applications of trawl processes}
Trawl processes constitute a class of stationary infinitely divisible stochastic processes and can be used in various applications. E.g.~we already pointed out above that they could be used for modelling stochastic volatility or intermittency.
In a recent article by \cite{BNLSV2012}, integer-valued trawl processes have been used to model count data or integer-valued data which are serially dependent. Speaking generally,  trawl processes can be viewed as a flexible class of stochastic processes which can be used to model stationary time series data, where the marginal distribution and the autocorrelation structure can be modelled independently from each other.


 \section{Tempo-spatial stochastic volatility/intermittency}
\label{Section:SpecVol}
Stochastic volatility/intermittency plays a key role in various applications including turbulence and finance. While a variety of purely temporal stochastic volatility/intermittency models can be found in the literature, suitable 
tempo-spatial stochastic volatility/intermittency models still need to be developed.

Volatility modulation within the framework of an ambit field can be achieved by four complementary  methods: By introducing a stochastic integrand (the term  $\sigma_t({\bf x})$ in the definition of an ambit field), 
or by (extended) subordination
or by probability mixing or \Levy mixing.
We will discuss all four methods in the following.

\subsection{Volatility modulation via a stochastic integrand}\label{SubSection:VolViaIntegrand}
Stochastic volatility in form of a stochastic integrand has already been included in the initial  definition of an ambit field, see \eqref{ambfi}.
The interesting aspect, which we have not addressed yet, is  how a model for the stochastic volatility field $\sigma_t(x)$ can be specified in practice. We will discuss several relevant choices in more detail in the following.

There are essentially two approaches which can be used for constructing a relevant stochastic volatility field: Either one specifies the stochastic volatility field directly as a random field (e.g.~as another ambit field), or one starts from a purely temporal (or spatial) stochastic volatility process and then  generalises the stochastic process to a random field in a suitable way.
In the following, we will present examples of both types of construction.

\subsubsection{Kernel-smoothing of a \Levy basis}
First, we focus on the modelling approach where we directly specify a random field for the volatility field.
A natural starting point for modelling the volatility is given by kernel-smoothing of a homogeneous \Levy basis -- possibly combined with a (nonlinear) transformation to ensure positivity. 
For instance, let
\begin{align}\label{Vol1}
\sigma^2_t({\bf x}) = V\left(\int_{\R_t\times \R^d}
j({\bf x}, {\boldsymbol \xi}, t-s) L^{\sigma}(d{\boldsymbol \xi}, ds)\right),
\end{align}
where $L^{\sigma}$ is a homogeneous \Levy basis independent of $L$, $j:\R^{d+1+d}\mapsto \R_+$ is an integrable kernel function satisfying $j({\bf x}, u, {\boldsymbol \xi})=0$ for $u< 0$  and
 $V: \R \to \R_+$ is a continuous, non-negative  function.  
 Note that $\sigma^2$ defined by (\ref{Vol1})  is stationary in the temporal dimension.
As soon as $j({\bf x}, {\boldsymbol \xi}, t-s)= j^*({\bf x}-{\boldsymbol \xi}, t-s)$ for some function $j^*$ in \eqref{Vol1}, then the stochastic volatility is both stationary in time and homogeneous in space.

Clearly, the kernel function $j$ determines the tempo-spatial autocorrelation structure of the volatility field.

Let us discuss some examples next.
\begin{ex}[Tempo-spatial trawl processes]
Suppose  the kernel function is given by 
\begin{align*}
j({\bf x},{\boldsymbol \xi}, t-s) =  {\bf I}_{A^{\sigma}}({\boldsymbol \xi}-{\bf x}, s-t),
\end{align*}
where $A^{\sigma}\subset \R^d\times (-\infty, 0]$. Further, let $A_t^{\sigma}({\bf x}) = A^{\sigma} + ({\bf x},t)$; hence $A_t^{\sigma}({\bf x})$ is a homogeneous and nonanticipative ambit set.  Then 
\begin{align*}
\sigma_t^2({\bf x})=V\left(\int_{\R_t\times \R^d}
{\bf I}_A^{\sigma}({\boldsymbol \xi}-{\bf x}, s-t) L^{\sigma}(d{\boldsymbol \xi}, ds)\right)=V\left(\int_{A_t^{\sigma}({\bf x})}
 L^{\sigma}(d{\boldsymbol \xi}, ds)\right)
= V\left(L^{\sigma}(A^{\sigma}_t({\bf x}))\right).
\end{align*}
Note that the random field $L^{\sigma}(A^{\sigma}_t({\bf x}))$ can be regarded as a \emph{tempo-spatial trawl process}.
  \end{ex}
\begin{ex}
Let 
\begin{align*}
j({\bf x},{\boldsymbol \xi},t-s) = j^*({\bf x},{\boldsymbol \xi},t-s){\bf I}_{A^{\sigma}({\bf x})}({\boldsymbol \xi}, s-t),
\end{align*}
for an integrable kernel function $j^*$ and 
where $A^{\sigma}({\bf x})\subset \R^d\times (-\infty, 0]$. Further, let $A_t^{\sigma}({\bf x}) = A^{\sigma}({\bf x}) + ({\bf 0},t)$.
Then
\begin{align}\label{GenVol}
\sigma^2_t({\bf x}) = V\left(\int_{A_t^{\sigma}({\bf x})}
j({\boldsymbol \xi}, {\bf x}, t-s) L^{\sigma}(d{\boldsymbol \xi}, ds)\right),
\end{align}
which is a transformation of an ambit field (without stochastic volatility).
\end{ex}

 Let us look at some more concrete examples for the stochastic volatility field.
 \begin{ex}
A rather simple specification is given by choosing $L^{\sigma}$ to be a standard normal \Levy basis and $V(x)=x^2$. 
Then $\sigma_s^2({\boldsymbol \xi})$ would be positive and pointwise $\chi^2$-distributed with one degree of freedom.
\end{ex}
\begin{ex} One could also work with a general   \Levy basis, in particular Gaussian,  and $V$ given by the exponential function, see e.g. \cite{BNSch04} and \cite{SBNE2005}.
\end{ex}
\begin{ex} A non-Gaussian example would be to choose $L^{\sigma}$ as an inverse Gaussian \Levy basis and  $V$ to be the identity function.
 \end{ex}
 
 \begin{ex}
 We have already mentioned that the
 kernel function $j$  determines the autocorrelation structure of the volatility field. E.g.~in the absence of spatial correlation 
 one could start off with the choice   $j({\bf x},t-s, {\boldsymbol \xi}) = \exp(-\lambda (t-s))$ for $\lambda > 0$ mimicking the Ornstein-Uhlenbeck-based stochastic volatility models, see e.g.~\cite{BNSch04}.
\end{ex}

\subsubsection{Ornstein-Uhlenbeck volatility fields}
Next, we show how to construct a stochastic volatility field by extending a stochastic process by a spatial dimension.
Note that our objective is to construct a stochastic volatility field which is stationary (at least in the temporal direction). 
Clearly, there are many possibilities on how this can be done and we focus on a particularly relevant one in the following, namely the \emph{Ornstein-Uhlenbeck-type volatility field} (OUTVF). The choice of using an OU process as the  stationary base component is motivated by the fact that non-Gaussian OU-based stochastic volatility models, as e.g.~studied in \cite{BNS2001a}, are analytically tractable and tend to perform  well in practice, at least in the purely temporal case. In the following, we will restrict our attention to the case $d=1$, i.e.~that the spatial dimension is one-dimensional.

Suppose now that $\tilde{Y}$ is a positive OU type process with rate parameter $\lambda > 0$
and generated by a L\'{e}vy subordinator $Y$, i.e.%
\begin{equation*}
\tilde{Y}_{t}=\int_{-\infty }^{t}e^{-\lambda \left( t-s\right) }\mathrm{d}%
Y_{s},
\end{equation*}%
 We call a  stochastic volatility field $\sigma _{t}^{2}\left( x\right)$ on  $\mathbb{R}\times \mathbb{R}$ an
\emph{Ornstein-Uhlenbeck-type volatility field} (OUTVF), if it is defined as follows
\begin{equation}
\tau _{t}\left( x\right) =\sigma _{t}^{2}\left( x\right) =e^{-\mu x}\tilde{Y}%
_{t}+\int_{0}^{x}e^{-\mu \left( x-\xi \right) }\mathrm{d}Z_{\xi |t},
\label{voldef}
\end{equation}%
 where  $\mu > 0$ is the spatial rate parameter
 and 
where 
$\mathcal{Z}=\left\{ Z_{\cdot |t}\right\} _{t\in \mathbb{R}_{+}}$ is a
family of L\'{e}vy processes, which we define more precisely in the next but one paragraph.

Note  that in the above construction, we start from an OU process in time. In particular,  $\tau_t(0)$ is an OU process. The spatial structure is then introduced 
by two components:
First, we we add an exponential weight $e^{-\mu x}$ in the spatial direction, which reaches its maximum for $x=0$ and decays the further away we get from the purely temporal case.
Second, an integral is added which resembles an OU-type process in the spatial variable $x$. However, note here that the integration starts from $0$ rather than from $-\infty$,  and hence the resulting component is not stationary in the spatial variable $x$. (This could be changed if required in a particular
 application.)

Let us now focus in more detail on how to define the family of \Levy processes $\mathcal{Z}$.
Suppose  $\tilde{X}=\left\{ \tilde{X}_{t}\right\} _{t\in \mathbb{R}}$\ is  a
stationary, positive and infinitely divisible process on $\mathbb{R}$.
Next we define 
 $Z_{|\cdot }=\left\{ Z_{x|\cdot }\right\} _{x\in \mathbb{R}_{+}}$\ as the so-called \emph{L\'{e}vy supra-process} generated by $\tilde{X}$, that is $\left\{
Z_{x|\cdot }\right\} _{x\in \mathbb{R}_{+}}$\ is a family of stationary
processes such that $Z_{|\cdot }$ has independent increments, i.e. for any $%
0<x_{1}<x_{2}<\cdots <x_{n}$ the processes $Z_{x_{1}|\cdot },Z_{x_{2}|\cdot
}-Z_{x_{1}|\cdot },\dots,Z_{x_{n}|\cdot }-Z_{x_{n-1|\cdot }}$ are mutually
independent, and such that for each $x$\ the cumulant functional of $%
Z_{x|\cdot }$\ equals $x$ times the cumulant functional of $\tilde{X}$, i.e. 
\begin{align*}\mathrm{C}\{m\ddagger Z_{x|\cdot }\}=x\mathrm{C}\{m\ddagger \tilde{X}\},
\end{align*}
where%
\begin{equation*}
\mathrm{C}\{m\ddagger \tilde{X}\}=\log \mathrm{E}\left\{ e^{im\left( \tilde{X%
}\right) }\right\},
\end{equation*}%
with 
$m\left( \tilde{X}\right) =\int \tilde{X}_{s}m\left( \mathrm{d}s\right)$, $m$\ denoting an \lq arbitrary\rq\ signed measure on $\mathbb{R}$. Then at any $%
t\in \mathbb{R}$ the values $Z_{x|t}$\ of $Z_{\cdot |\cdot }$\ at time $t$\
as $x$\ runs through $\mathbb{R}_{+}$\ constitute a L\'{e}vy process that we
denote by $Z_{\cdot |t}$. This is the L\'{e}vy process occurring in the
integral in (\ref{voldef}).

Note that $\tau $\ is stationary in $t$ and that $\tau _{t}\left(
x\right) \rightarrow \tilde{Y}_{t}$ as $x\rightarrow 0$.

\begin{ex}
Now suppose, for simplicity, that $\tilde{X}$\ is an OU process with rate
parameter $\kappa $ and generated by a L\'{e}vy process $X$. Then%
\begin{align*}
\mathrm{Cov}\{\tau _{t}\left( x\right) ,\tau _{t^{\prime }}\left( x^{\prime
}\right) \}=\frac{1}{2}\left( \mathrm{Var}\{Y_{1}\}\lambda ^{-1}e^{-\lambda
\left( \left\vert t-t^{\prime }\right\vert \right) -\mu \left( x+x^{\prime
}\right) }+\mathrm{Var}\{\tilde{X}_{0}\}\mu ^{-1}e^{-\kappa |t-t^{\prime
}|-\mu \left\vert x-x^{\prime }\right\vert }\right.\\
\left.-\mathrm{Var}\{\tilde{X}%
_{0}\}\mu ^{-1}e^{-\kappa \left\vert t-t^{\prime }\right\vert -\mu \left(
x+x^{\prime }\right) }\right).
\end{align*}%
If, furthermore, $\mathrm{Var}\{Y_{1}\}=\mathrm{Var}\{\tilde{X}_{0}\}$ and $%
\kappa =\lambda =\mu $\ then for fixed $x$\ and $x^{\prime }$\ the
autocorrelation function of $\tau $\ is%
\begin{equation*}
\mathrm{Cor}\{\tau _{t}\left( x\right) ,\tau _{t^{\prime }}\left( x^{\prime
}\right) \}=e^{-\kappa |t-t^{\prime }|}e^{-\kappa \left\vert x-x^{\prime
}\right\vert }.
\end{equation*}
\end{ex}

This type of construction can of course be generalised in a variety of ways,
including dependence between $X$\ and $Y$\ and also superposition of OU
processes.

Note that the process $\tau_t(x)$ is in general not predictable, which is disadvantageous given that  we want to construct Walsh-type stochastic integrals. However, if we choose $\widetilde X$ to be of OU type, then we obtain a predictable stochastic volatility process.

\subsection{Extended subordination and meta-times}
An alternative way of volatility modulation  is by means of (extended) subordination.
Extended subordination and meta-times are important concepts in the ambit framework, which have recently been introduced by \cite{BN2010} and \cite{BNPedersen2010}, and we will review their main results in the following.
Note that  \emph{extended subordination} generalises the classical concept of subordination of L\'{e}vy
processes to \emph{subordination of L\'{e}vy bases}.
This in turn will be based on a concept of \emph{meta-times}.

\subsubsection{Meta-times}
This section reviews the concept
of meta-times, which we will  link to the idea of extended subordination in the following section.

\begin{defi} Let $S$ be a Borel set in $\mathbb{R}%
^{k} $. A \emph{meta-time} $\mathbf{T}$ on $%
S$ is a mapping from $\mathcal{B}_{b}(S)$ into $%
\mathcal{B}_{b}(S)$ such that 
\begin{enumerate}
\item $\mathbf{T}\left( A\right) $ and $\mathbf{T}\left( B\right) $ are
disjoint whenever $A,B\in \mathcal{B}_{b}(S)$ are disjoint.
 \item $\mathbf{T}\left( \cup _{j=1}^{\infty }A_{j}\right) =\cup
_{j=1}^{\infty }\mathbf{T}\left( A_{j}\right) $ whenever $A_{1},A_{2},\dots\in 
\mathcal{B}_{b}(S)$ are disjoint and $\cup _{j=1}^{\infty
}A_{j}\in \mathcal{B}_{b}(S)$.
\end{enumerate}
\end{defi}
A slightly more general definition is the following one.
\begin{defi} Let $S$\ be a Borel set in $\mathbb{R}%
^{k} $. A \emph{full meta-time} $\mathbf{T}$ on $S$\ is a mapping from $\mathcal{B}(S)$\ into $\mathcal{B}%
(S)$\ such that 
\begin{enumerate}
\item $\mathbf{T}\left( A\right) $ and $\mathbf{T}\left( B\right) $ are
disjoint whenever $A,B\in \mathcal{B}(S)$ are disjoint.
\item $\mathbf{T}\left( \cup _{j=1}^{\infty }A_{j}\right) =\cup
_{j=1}^{\infty }\mathbf{T}\left( A_{j}\right) $ whenever $A_{1},A_{2},\dots\in 
\mathcal{B}(S)$ are disjoint.
\end{enumerate}
\end{defi}
Suppose $\mu $\ is a measure on $\left( S,\mathcal{B}(S)\right) $ and let $\mathbf{T}$\ be a full meta-time on $S$. 
Define $\mu \left( \cdot \curlywedge \mathbf{T}\right) $ as the mapping from
\ $\mathcal{B}(S)$ into $\mathbb{R}$ given by $\mu \left(
A\curlywedge \mathbf{T}\right) =\mu \left( \mathbf{T}\left( A\right) \right) 
$ for any $A\in \mathcal{B}(S)$. Then $\mu \left( \cdot
\curlywedge \mathbf{T}\right) $ is a measure on $\left( S,\mathcal{%
B}(S)\right)$.
We speak of $\mu \left( \cdot \curlywedge \mathbf{T}\right) $ as the \emph{%
subordination of} $\mu $ \emph{by} $\mathbf{T}$.
Similarly, if $\mathbf{T}$ is a meta-time on $S$ we speak of $\mu
\left( A\curlywedge \mathbf{T}\right) =\mu \left( \mathbf{T}\left( A\right)
\right) $ for $A\in \mathcal{B}_{b}(S)$\ as the subordination of $%
\mu $\ by $\mathbf{T}$.

Let us now recall an important result, which says that any  measure, which is finite on compacts, can be represented as the image measure of the Lebesgue measure of a meta-time.
\begin{lem}{\citet[Lemma 3.1]{BNPedersen2010}}\label{ImOfLeb}
Let $T$\ be a measure on $\mathcal{B}(S)$\
satisfying $T\left( A\right) <\infty $ for all $A\in \mathcal{B}_{b}(%
S)$ \ Then there exists a meta-time $\mathbf{T}$\ such that, for
all $A\in \mathcal{B}_{b}(S)$, we have 
\begin{equation*}
T\left( A\right) =\mathrm{Leb}\left( \mathbf{T}\left( A\right) \right) .
\end{equation*}%
We speak of $\mathbf{T}$ as a \emph{meta-time} associated to $T$.
\end{lem}

\begin{rem}$\mathbf{T}$\ can be chosen so that $\mathbf{T}\left(
A\right) =\left\{ \mathbb{T}\left( {\bf x}\right) :{\bf x}\in A\right\} $ where $\mathbb{%
T}$ is the inverse of a measurable mapping $\mathbb{U}: S\rightarrow S$ determined from $\mathbf{T}$. Here integration under
subordination satisfies

\begin{equation*}
\int_{S}f\left( s\right) \mu \left( \mathrm{d}s\curlywedge
\mathbf{T}\right) =\int_{S}f\left( s\right) \mu \left( \mathbf{T}\left( 
\mathrm{d}s\right) \right) =\int_{S}f\left( \mathbb{U}\left(
u\right) \right) \mu \left( \mathrm{d}u\right) .
\end{equation*}
\end{rem}

Suppose $S$ is open and that $T$\ is a measure on $S$\
which is absolutely continuous with density $\tau$. If we can find a
mapping $\mathbb{T}$\ from $S$\ to $S$\ sending Borel
sets into Borel sets and such that the Jacobian of $\mathbb{T}$ exists and
satisfies%
\begin{equation*}
|\mathbb{T}_{/{\bf x}}|=\tau \left( {\bf x}\right),
\end{equation*}%
then $\mathbf{T}$ given by $\mathbf{T}\left( A\right) =\left\{ \mathbb{T}%
\left( {\bf x}\right) :{\bf x}\in A\right\} $, is a natural choice of meta-time induced
by $T$.
In fact, by the change of variable formula,%
\begin{equation*}
\leb\left( \mathbf{T}\left( A\right) \right) =\int 1_{\mathbf{T}%
\left( A\right) }(y)\mathrm{d}y=\int 1_{A}\left( {\bf x}\right) |\mathbb{T}_{/{\bf x}}|%
\mathrm{d}{\bf x}=\int_{A}\tau \left( {\bf x}\right) \mathrm{d}{\bf x}=T\left( A\right) ,
\end{equation*}%
verifying that $\mathbf{T}$\ is a meta-time associated to $T$.

\begin{rem}
$\mathbf{T}$ is not uniquely determined by $T$ and two
different meta-times associated to $T$ may yield different subordinations of
one and the same measure $\mu$.
\end{rem}

\subsubsection{Extended subordination of Levy bases}
Let us now study the concept of subordination in the case where we deal with a \Levy basis.
Let $T$ be a full random measure on $S$. Then, by 
Lemma \ref{ImOfLeb}, there exists a.s. a random meta-time $\mathbf{T}$ determined by $T$  and with the property that 
$\leb\left( \mathbf{T}\left( A\right) \right) =T\left( A\right)$,
for all $A\in \mathcal{B}_{b}(S)$. There are two cases  of particular interest to us. 
First,
 $\mathbf{T}$ is induced by a L\'{e}vy basis $L$ on $S$ that
is non-negative, dispersive and of finite variation. 
Second,  $\mathbf{T}$ is induced by an absolutely continuous random measure $T$
on $S$\ with a non-negative density $\tau $ satisfying $T\left(
A\right) =\int_{A}\tau \left( {\bf z}\right) \mathrm{d}{\bf z}<\infty $\ for all $A\in 
\mathcal{B}_{b}(S)$.
\begin{defi}
Let $L$ be a L\'{e}vy basis on $S$ and
let $T$, independent of $L$, be a full random measure on $S$. The
\emph{extended subordination} of $L$
 by $T$ is the random measure $L(\cdot \curlywedge T)$
defined by\ 
\begin{equation*}
L(A\curlywedge T)=L\left( \mathbf{T}\left( A\right) \right)
\end{equation*}%
for all $A\in \mathcal{B}_{b}\left( S\right) $\ and where $\mathbf{%
T}$\ is a meta-time induced by $T$ (in which case $\leb\left( 
\mathbf{T}\left( A\right) \right) =T\left( A\right) $).
\end{defi}
Note that the L\'{e}vy basis $L$\ may be $n$-dimensional.
We shall occasionally write $L\curlywedge T$ for $L\left( \cdot \curlywedge
T\right) $, which is   a random measure on $(%
S,\mathcal{B}_{b}(S))$.

In the case that $L$ is a homogeneous \Levy basis and 
since  $T$ is assumed to be independent of $L$,\ $L(\cdot \curlywedge T)$
is a (in general not homogeneous) L\'{e}vy basis, whose conditional cumulant
function satisfies%
\begin{equation}\label{CumExtSub}
\mathrm{C}\{\zeta \ddagger L(A\curlywedge T)|T\}=T(A)\mathrm{C}\{\zeta
\ddagger L^{\prime }\}
\end{equation}%
for all $A\in \mathcal{B}_{b}(\mathcal{S})$ and where $L^{\prime }$\ is the L%
\'{e}vy seed of $L$.
On a distributional level one may, without attention to
the full probabilistic definition of $L(\cdot \curlywedge T)$ presented
above, carry out many calculations purely from using the identity established in (\ref{CumExtSub}).

\subsubsection*{Remarks}

 The two key formulae $L(\cdot \curlywedge T)=L\left( \mathbf{T}\left(
\cdot \right) \right) $ and $\leb\left( \mathbf{T}\left( A\right)
\right) =T\left( A\right) $\ show that the concepts of extended
subordination and meta-time together generalise the classical subordination
of L\'{e}vy processes.
 Provided that $L$ is homogeneous we have that 
\begin{equation*}
\mathrm{C}\{\zeta \ddagger L(A\curlywedge T)|T\}=\leb\left( \mathbf{T%
}\left( A\right) \right) \mathrm{C}\{\zeta \ddagger L^{\prime }\},
\end{equation*}%
and hence%
\begin{equation*}
\mathrm{C}\{\zeta \ddagger L(A\curlywedge T)\}=\log \mathrm{E}\left\{
e^{T\left( A\right) \mathrm{C}\{\zeta \ddagger L^{\prime }\}}\right\}.
\end{equation*}
Hence we can deduce the following results.
\begin{itemize}
\item The values\ of the subordination $L(\cdot \curlywedge T)$\ of $L$\ are
infinitely divisible provided the values of $T$\ are infinitely divisible
and $L$\ is homogeneous.

\item If $L$\ is homogeneous and if $T$\ is a homogeneous L\'{e}vy\ basis
then $L(\cdot \curlywedge T)$\ is a homogeneous L\'{e}vy basis.

\item In general, $\mathbf{T}$\ is not uniquely determined by $T$.
Nevertheless, provided the L\'{e}vy basis $L$\ is homogeneous the law of $%
L(\cdot \curlywedge T)$ does not depend on the choice of meta-time $\mathbf{T%
}$.
\end{itemize}

\subsubsection*{L\'{e}vy-It\^{o} type representation of $L\left( \cdot \curlywedge
\cdot \right) $}
We have already reviewed the L\'{e}vy-It\^{o} representation for a 
 dispersive L\'{e}vy basis $L$, see (\ref{LevyIto}).
It follows directly that the subordination of $L$\ by the random measure $T$%
\ with associated meta-time $\mathbf{T}$\ has a L\'{e}vy-It\^{o} type
representation
\begin{align*}
L\left( A\curlywedge T\right) &=
a\left( \mathbf{T}\left( A\right) \right)
+\sqrt{b(\mathbf{T}\left( A\right))}W\left( \mathbf{T}\left( A\right) \right) \\
 &+\int_{\{\left\vert y\right\vert >1\}}y N\left( dy;\mathbf{T}\left(
 A\right) \right) 
 +\int_{\{\left\vert y \right\vert \leq 1\}} y\left( N-\nu\right)
 \left( dy;\mathbf{T}\left( A\right) \right) .
\end{align*}

\subsubsection*{L\'{e}vy measure of $L\curlywedge T$}
Suppose for simplicity that $L\curlywedge T$\ is non-negative. According to  \cite{BN2010}, the \Levy measure $\widetilde \nu$ of $L\curlywedge T$ is related to the L\'{e}vy measure $\nu$ of $T$\ by
\begin{equation*}
\widetilde \nu \left( \mathrm{d}x; s\right) =\int_{%
\mathbb{R}}P\left(  L^{\prime }_y(s) \in dx\right) \nu \left( 
dy; s\right).
\end{equation*}

\subsubsection{Extended subordination and volatility}
In the context of ambit stochastics one considers volatility fields $\sigma $
in space-time, typically specified by the squared field $\tau _{t}\left(
{\bf x}\right) =\sigma _{t}^{2}\left( {\bf x}\right) $.
 So far, we have used a stochastic volatility random field $\sigma$ in the integrand of the ambit field to introduce volatility modulation.
 
A complementary method consists of introducing stochastic volatility by extended subordination.
The volatility is incorporated in the modelling through a meta-time associated to the measure $T$\ on $S$\ and given by 
\begin{equation*}
T\left( A\right) =\int_{A}\tau _{t}\left( {\bf x}\right) \mathrm{d}{\bf x}\mathrm{d}t
\end{equation*}
A natural choice of meta-time is $\mathbf{T}\left( A\right) =\left\{ \mathbb{%
T}\left( {\bf x},t\right) :\left( {\bf x},t\right) \in A\right\}$, where $\mathbb{T}$\
is the mapping given by%
\begin{equation*}
\mathbb{T}\left( {\bf x},t\right) =\left( {\bf x},\tau _{s}^{+}\left( {\bf x}\right) \right)
\end{equation*}%
and where%
\begin{equation*}
\tau _{t}^{+}\left( {\bf x}\right) =\int_{0}^{t}\tau _{s}\left( {\bf x}\right) \mathrm{d}%
s.
\end{equation*}

The above construction of a meta-time in a tempo-spatial model is very general.
One can  construct a variety of models for the random field $\tau_t({\bf x})$, which lead to new model specifications. Essentially, this leads us back to the problem we tackled in the previous Subsection, where we discussed how such fields can be modelled. For instance, one could model $\tau_t({\bf x})$ by an Ornstein-Uhlenbeck type volatility field or any other model discussed in  
 Subsection \ref{SubSection:VolViaIntegrand}. Clearly, the concrete choice of the model needs to be tailored to the particular application one has in mind.

\subsection{Probability mixing and \Levy mixing}
Volatility modulation can also be obtained through 
 probability mixing as well as \Levy mixing.

The main idea behind the concept of 
probability or distributional mixing is to construct new distributions by randomising a parameter from a given parametric distribution.
\begin{ex}\label{ProbMixEx}
Consider our very first base model \eqref{base}. Now suppose that the corresponding \Levy basis is homogeneous and Gaussian, i.e.~the corresponding \Levy seed is given by 
$L'\sim N(\mu+ \beta \sigma^2,\sigma^2)$ with $\mu, \beta \in \R$ and $\sigma^2 > 0$. Now we use probability mixing and suppose that in fact $\sigma^2$ is random. 
Hence, the conditional law of the \Levy seed is given by $L'|\sigma\sim N(\mu+ \beta \sigma^2,\sigma^2)$.
 Due to the scaling property of the Gaussian distribution, such a model can  be represented as \eqref{ambfi} and, hence, in this particular case probability mixing and stochastic volatility via a stationary stochastic integrand essentially have the same effect. 
Suppose that the conditional variance $\sigma^2$ has a generalised inverse Gaussian (GIG) distribution rather than being just a constant, then $L'$ follows in fact a generalised hyperbolic (GH) distribution. Such a construction falls into the class of normal variance-mean  mixtures.
\end{ex}

In this context it is important to note that  probability/distributional mixing does not generally lead to infinitely divisible distributions, see e.g.~\citet[Chapter VI]{SteutelvanHaarn2004}.
Hence
 \cite{BNPAT2012} propose to work with  
\emph{\Levy mixing} instead of probability/distributional mixing. \Levy mixing is a method which (under mild conditions) leads to classes of infinitely divisible distributions again.
Let us review the main idea behind that concept.

Let $L$ denote a factorisable \Levy basis on $\R^k$ with CQ $(a, b, \nu(dx),c)$. Suppose that the \Levy measure $\nu(dx)$ depends on a possibly multivariate parameter ${\boldsymbol \theta}\in \Theta$, say, where $\Theta$ denotes the parameter space. In that case, we write  $\nu(dx;{\boldsymbol \theta})$.
Then, the generalised \Levy measure of $L$ is given by 
$\nu(dx; {\boldsymbol \theta}) c(d{\bf z})$.
Now let $\gamma$ denote a measure on $\Theta$ and define
\begin{align*}
\tilde n(dx, d{\bf z}) = \int_{\Theta} \nu(dx; {\boldsymbol \theta}) \gamma(d{\boldsymbol \theta})c(d{\bf z}),
\end{align*}
where we assume that 
\begin{align}\label{GammaCond}
\int_{\R}(1 \wedge x^2) \tilde n(dx, d{\bf z}) < \infty.
\end{align}
Then there exists a \Levy basis $\tilde L$ which has $\tilde n$ as its generalised \Levy measure. We call the \Levy basis $\tilde L$ the \Levy basis obtained by \levy-mixing $L$ with the measure $\gamma$.

Let us study a concrete example of \Levy mixing in the following.
\begin{ex}
Suppose $L$ is a homogeneous  \Levy basis with CQ given by $(0,0, \nu(dx), \leb)$ with $\nu(dx) = \theta \delta_1(dx)$ for $\theta > 0$. I.e.~the corresponding \Levy seed is given by  $L'\sim Poi(\theta)$. Now we do a \levy-mixing of the intensity parameter $\theta$. Let
\begin{align*}
\tilde n(dx, d{\bf z}) = \int_{\Theta} \theta  \gamma(d\theta) \delta_1(dx)d{\bf z},
\end{align*}
for a measure $\gamma$ satisfying condition \eqref{GammaCond}. Let $\tilde L$ be the \Levy basis with CQ $(0,0, \int_{\Theta} \theta  \gamma(dv) \delta_1(dx), \leb)$. In that case, the base model \eqref{base} would be transferred into a model of the form
\begin{align*}
\int_{\Theta}\int_{A_t({\bf x})} h({\bf x},t; {\boldsymbol \xi}, s)  \tilde L(d{\boldsymbol \xi},ds,dv).
\end{align*}
\end{ex}
\begin{ex}
Let us consider the example of a (sup)OU process, see \cite{BN2000}, \cite{BNStelzer2010b} and \cite{BNPAT2012}. 
Let $L$ denote a subordinator with \Levy measure $\nu_L$ (and without drift) and consider an OU process
$$Y_t = \int_{-\infty}^t e^{-\theta (t-s)} dL(s),\qquad \theta > 0.$$
A straightforward computation leads to the following expression for its cumulant function (for $\zeta \in \R$):
\begin{align*}
C\{\zeta \ddagger Y_t\} = \int_0^{\infty}\left(e^{i \zeta x}-1\right) \nu(dx; \theta),
\quad \text{ where } \quad 
\nu(dx, \theta) = \int_0^{\infty}\nu_L(e^{\theta u} dx) du
\end{align*}
 is a mixture of $\nu_L$ with the Lebesgue measure. 
A \Levy mixing can be carried out with respect to the parameter $\theta$, i.e.
\begin{align*}
\tilde \nu(dx)= \int_{0}^{\infty}\nu(dx, \theta)\gamma(d\theta),
\end{align*}
where $\gamma$ is a measure on $[0,\infty)$ satisfying $\int_0^{\infty}x \tilde \nu(dx)
<\infty$. Then $\tilde \nu$ is a \Levy measure again. Now, let $\tilde L$ be the \Levy basis with extended \Levy measure 
$$
\nu_L(dx)du \gamma(d\theta),
$$
and define 
the  supOU process $\tilde Y_t$ w.r.t.~$\tilde L$ by 
\begin{align*}
\tilde Y_t= \int_0^{\infty}\int_{-\infty}^t e^{-\lambda (t-s)} \tilde L(ds, d\lambda).
\end{align*}
Then 
the cumulant function of $\tilde Y_t$ is given by 
\begin{align*}
C\{\zeta \ddagger \tilde Y_t\} &= 
 \int_0^{\infty} \int_{-\infty}^t C\{ \zeta e^{-\theta (t-s)} \ddagger L_1\} ds \gamma(d\theta)
=
\int_0^{\infty}\int_0^{\infty} \left(
e^{i \zeta e^{-\theta u}}-1 \right)\nu_L(dx) du \gamma(d\theta)
\\
&=
\int_0^{\infty}\left(e^{i\zeta x}-1\right)\tilde \nu(dx).
\end{align*}
Hence, we have seen that a supOU process can be obtained from an OU process through \levy-mixing.
\end{ex}

\subsection{Outlook on volatility estimation}
Once a model is formulated and data are available the question of assessment of  volatility arises and while we do not discuss this in any detail, tools for this are available for some special classes of ambit processes, see \cite{BNCP2009a,BNCP09,BNCP10b} and also \cite{BNGraversen2011}. Extending these results to general ambit fields and processes is an interesting direction for future research.

\section{Conclusion and outlook}\label{Section:Conclusion}
  In this paper, we have given an overview of some of the main findings in ambit stochastics up to date, including a suitable stochastic integration theory, and have established new results on general properties of ambit field. 
The new results include sufficient conditions which ensure the smoothness of ambit fields. 
Also, 
we have formulated sufficient conditions which guarantee that an ambit field is a semimartingale in the temporal domain. 
 Moreover,  the concept of tempo-spatial stochastic  volatility/intermittency   within ambit fields has been further developed. Here our focus has been on four methods for volatility modulation: Stochastic scaling,  stochastic time change and extended subordination of random measures, and   probability and \Levy mixing of the volatility/intensity parameter.
 
 Future research will focus on applications of  the general classes of models developed in this paper in various fields, including empirical research on turbulence modelling as well as modelling e.g.~the term structure of interest rates in finance by ambit fields. In this context, it will be important to establish a suitable  estimation theory for general ambit fields as well as inference techniques.

\section*{Acknowledgement}
Financial support by the Center for Research in Econometric Analysis of Time
Series, CREATES, funded by the Danish National Research Foundation is gratefully acknowledged by O.~E.~Barndorff-Nielsen.
F.~E.~Benth acknowledges financial support from the Norwegian Research Council through the project "Energy markets: modelling, optimization and simulation" (EMMOS), eVita 205328. 
A.~E.~D.~Veraart acknowledges financial support by CREATES and  by a Marie Curie
FP7 Integration Grant within the 7th European Union Framework Programme.


\bibliographystyle{agsm}
\bibliography{EnergyBib}

\end{document}